\documentclass[final]{dmtcs-episciences}

\usepackage[utf8]{inputenc}

\usepackage[english]{babel}

\usepackage{amsmath}
\usepackage{amsfonts,amssymb,latexsym,graphics,amsthm}
\usepackage{hyperref}
\usepackage{tikz}
\usepackage{youngtab}
\usepackage{rotating}
\usepackage{url}

\renewenvironment{itemize}{\begin{list}{\labelitemi}{\leftmargin=1.5em}}{\end{list}}
\renewcommand{\labelitemi}{--}

\newcommand{\n}{\null}
\renewcommand{\b}{\mbox{$\bullet$}}

\author{Matthieu Josuat-Vergès  \affiliationmark{1}  }

\title[Stammering tableaux]{Stammering tableaux}

\affiliation{CNRS, Laboratoire d'Informatique Gaspard Monge, Université Paris-est Marne-la-Vallée, FRANCE}
\newtheorem{theo}{Theorem}[section]

\newtheorem{lemm}[theo]{Lemma}
\newtheorem{prop}[theo]{Proposition}

\theoremstyle{definition}
\newtheorem{defi}[theo]{Definition}
\newtheorem{rema}[theo]{Remark}


\received{2016-1-12}

\revised{2017-5-30, 2017-8-29}

\accepted{2017-8-29}

\publicationdetails{19}{2017}{3}{3}{1351}

\begin{document}

\maketitle

\Ylinethick{1.2pt}
\Yboxdim{8pt}

\begin{abstract}
The PASEP (Partially Asymmetric Simple Exclusion Process) is a probabilistic
model of moving particles, which is of great interest in combinatorics, since it was realized that its 
partition function counts a certain kind of tableaux. These tableaux have several variants such as permutations
tableaux, alternative tableaux, tree-like tableaux, Dyck tableaux, etc.
We introduce in this context certain walks in Young's lattice, that we call {\it stammering tableaux}
(by analogy with oscillating tableaux, vacillating tableaux, hesitating tableaux). Some natural bijections
make a link with rook placements in a double staircase, chains of Dyck paths obtained by successive addition
of ribbons, Laguerre histories, Dyck tableaux, etc.
\end{abstract}



\section{Introduction}


\subsection{The PASEP}

The PASEP (Partially Asymmetric Simple Exclusion Process) is a probabilistic model of moving particles
which motivated a lot of recent combinatorial developments. Our description of the model is very brief 
since we are more interested in the combinatorial objects, so we refer to \cite{derrida} for 
a definition, and \cite{corteel} for the related combinatorics. A state
in this model is a binary word of length $N$ over $\circ$ and $\bullet$, and can be interpreted as 
a sequence of $N$ sites, each being empty ($\circ$) or occupied by a particle ($\bullet$). The possible 
transitions are :
\begin{itemize}
 \item a factor $\bullet\circ$ becomes $\circ\bullet$ (a particle moves to the right),
 \item a factor $\circ\bullet$ becomes $\bullet\circ$ (a particle moves to the left),
 \item an initial $\circ$ becomes $\bullet$ (a new particle arrives from the left),  
 \item a final $\bullet$ becomes $\circ$ (a particle exits on the right).
\end{itemize}
These $4$ events occur with probabilities depending on $4$ real positive parameters $p$, $q$, $\alpha$, $\beta$, 
and by a rescaling argument we can assume $p=1$. For each state $\tau\in \{ \circ, \bullet \} ^N$, its 
probability in the stationary distribution is denoted $p_\tau$.  

A recipe to compute the stationary probability $p_\tau$ is the {\it Matrix Ansatz} of Derrida et al.~obtained 
in \cite{derrida}. Suppose that we have two operators $F$ and $E$ satisfying a commutation relation:
\begin{equation} \label{relEF}
   FE-qEF = F+E.
\end{equation}
We can identify a state $\tau$ with a binary word $w$ in $F$ and $E$ by the rule $\bullet \to F$, $\circ \to E$.
Using \eqref{relEF}, the resulting binary word can be put in normal form, i.e., a linear combination of $E^iF^j$:
\[
   w = \sum_{i,j \geq 0}  c_{i,j} E^i F^j
\]
where $c_{i,j}$ are polynomials in $q$ (and a finite number of them are nonzero).
Then, the result is as follows:
\[
   p_{\tau} = \frac{ \bar p_{\tau} }{Z_N}
\]
where the nonnormalized probability $\bar p_{\tau}$ is $\sum_{i,j \geq 0}  c_{i,j} \alpha^{-i} \beta^{-j}$,
and the normalizing factor is $Z_N = \sum \bar p_{\tau}$ summed over $\tau \in \{\bullet,\circ\}^N$
(it can also be obtained from the normal form of $(F+E)^N$ since this expands as the sum of all binary words).
A consequence is that $p_{\tau}$ and $Z_N$ are polynomials with nonnegative coefficients in 
$\alpha^{-1}$, $\beta^{-1}$, and $q$.

\subsection{Tableau combinatorics}

The connection of the PASEP model with permutations and statistics on them is done in \cite{corteel}, where
Corteel and Williams used so-called {\it permutation tableaux}, defined as 0-1 fillings of Young diagrams
with certain rules. Indeed, since each $\bar p_{\tau}$ is a polynomial, it is a natural to look for its combinatorial 
meaning, and the main result of \cite{corteel} shows that this quantity is a generating function of permutation
tableaux whose shape is a particular Young diagram depending on $\tau$.
In particular the partition function $Z_N$ is a 3-variate refinement of $(N+1)!$ with a combinatorial
interpretation on permutations.

Later, other kinds of tableaux have been introduced: {\it alternative tableaux} \cite{viennot}, 
{\it tree-like tableaux} \cite{aval1}, and {\it Dyck tableaux} \cite{aval2}. 
They are all variants of each other, but each has its own combinatorial properties.
For example, Viennot \cite{viennot} showed that the relation \eqref{relEF} leads to alternative tableaux,
bypassing the probabilistic model.
Tree-like tableaux \cite{aval1} have a nice tree structure, and there is a nice insertion algorithm that
permits building these object inductively.
Particularly relevant to the present work, Dyck tableaux \cite{aval2} connects with
labelled Dyck paths called {\it Laguerre histories} (see also \cite{josuat}).
These various refecences make clear that all this tableau combinatorics is interesting in itself,
and the link with the PASEP is sometimes not apparent.

Essentially, our {\it stammering tableaux} ({\it tableaux bégayants} in French) can be seen as another variant of these objects. 
However, the word ``tableaux'' has a different meaning here: they are not fillings of Young diagrams,
and should rather be considered as a variant of {\it oscillating tableaux, vacillating tableaux}, or {\it hesitating tableaux}
(see \cite{chen} and below). These objects are essentially obtained from the Matrix Ansatz in~\eqref{relEF}.

\subsection{Realizations of the Matrix Ansatz operators}

Although it is not {\it a priori} needed to compute the stationary probabilities $p_{\tau}$,
it is useful to have explicit realizations of operators $E$ and $F$ satisfying \eqref{relEF}.
For example, Derrida et al.~provide tridiagonal semi-infinite matrices realizations of $E$ and $F$
that were exploited in \cite{josuat}.

A natural idea to realize the relation in \eqref{relEF} is to start from the much more common relation
\begin{equation} \label{reldu}
  DU-qUD=I
\end{equation}
where $I$ is the identity (sometimes called the $q$-canonical commutation relation). 
Then, let  
\[
  F=D(U+I), \qquad E=D(U+I)U.
\]
It can be checked that the relation in \eqref{reldu} implies the one in \eqref{relEF}, indeed:
\begin{align*}
  FE &= D(U+I)D(U+I)U = D(U+I)DU + D(U+I)DUU, \\
  EF &= D(U+I)UD(U+I) = D(U+I)UD + D(U+I)UDU,
\end{align*}
and we get 
\[
  FE-qEF = D(U+I) + D(U+I)U = F+E.
\]
Thus, explicit realizations of $E$ and $F$ can be obtained from the ones of $D$ and $U$.

\subsection{Walks in Young's lattice}

In the particular case $q=1$, a classical way to realize the relation $DU-UD=I$ is to define 
\begin{equation} \label{duud}
  U(\lambda) = \sum_{\mu \gtrdot \lambda} \mu \qquad \text{ and } \qquad 
  D(\lambda) = \sum_{\mu \lessdot \lambda} \mu,
\end{equation}
where $\lessdot$ is the cover relation in Young's lattice $\mathcal{Y}$, and $D$, $U$ are seen 
as linear operators on the vector space based on $\mathcal{Y}$. 
We refer to Stanley's notion of differential poset \cite{stanley} (see next section).
In this setting, 
\begin{itemize} 
 \item {\it oscillating tableaux} naturally appear in the expansion of $(U+D)^n \varnothing $,
 \item {\it vacillating tableaux} naturally appear in the expansion of $( (U+I) (D+I) )^n \varnothing $,
 \item {\it hesitating tableaux} naturally appear in the expansion of $ ( DU + ID + UI )^n \varnothing $.
\end{itemize}
See \cite{chen} for definitions and more on this subject. So, motivated by the Matrix Ansatz of the PASEP, it is 
natural to consider the tableaux appearing in similar expansions, with the powers of $F+E=D(U+I)^2$.
These are our {\it stammering tableaux}, formally defined in the next section.

These new objects allow one to see more clearly some properties that have been investigated before
in permutation tableaux and their variants. The various bijections here are rather simple and natural.
In particular, there is the question of a recursive construction to make clear that there are 
$n!$ such objects of size $n$, which can involve insertion algorithms \cite{aval1,aval2}.
As we will see through this article, the recursive construction in our case is extremely simple.
Another property is the link with earlier results by Françon and Viennot in \cite{francon}.
We will see that this link becomes clear through our notion of chains of Dyck shapes
(see Section~\ref{sec:5}). But perhaps the most important point is that these objects
make a bridge between the tableaux combinatorics of the PASEP and the tableaux combinatorics
of Young's lattice.

\subsection{Organization}

Sections~\ref{sec:2} to~\ref{sec:6} give definitions and several bijections 
(except Section~\ref{sec:4} that shows that a poset of Dyck paths naturally appearing in
Section~\ref{sec:3} is a lattice).
In Section~\ref{sec:7}, we consider a variant of stammering tableaux that we know how to enumerate, 
and Section~\ref{sec:8} contains a few open problems.

\section*{Acknowledgement}

This work was first presented at the ``Journées Viennot'' in June 2012, at the University of Bordeaux.
I thank the organizers for the invitation, and I thank Xavier for all I learned from him.

Also, I thank Tom Roby for drawing my attention to the work in his thesis \cite{roby}.

This work is supported by ANR-12-BS01-0017 (CARMA).

\section{Stammering tableaux and the double staircase}

We begin with a few definitions and reference to the literature.
A general reference is \cite{stanley2}.

A {\it Young diagram} is a bottom left aligned set of unit squares, for example:
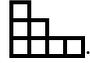
\[
  \yng(1,2,4).
\]
Our definition respects the French convention but others exist. Such a diagram is 
characterized by its sequence of row lengths, from bottom to top. The previous
example gives $4,2,1$. Such a sequence is an {\it integer partition}, i.e., a 
nonincreasing sequence of positive integers. This gives a bijection between 
Young diagram and integer partition.

Each of the unit square in the Young diagram is called a {\it cell}.  
Each cell has four corners, and a point being a corner of at least one cell
is called a {\it vertex} of the Young diagram.

Let $\mathcal{Y}$ denote Young's lattice, which is the set of all Young digrams endowed
with the inclusion order (seeing Young diagrams as set of cells).
Let $\varnothing$ denote its minimal element, the empty Young diagram. 
Let $\lessdot$ denote the cover relation on $\mathcal{Y}$ (and $\gtrdot$ the reverse relation).

An important property of $\mathcal{Y}$ is that it is a {\it differential poset}
\cite{stanley}.

\label{sec:2}

\begin{defi}
A {\it stammering tableau} of size $n$ is a sequence $(\lambda^{(0)}, \dots , \lambda^{(3n)}) \in \mathcal{Y}^{3n+1}$ 
such that $\lambda^{(0)}=\lambda^{(3n)}=\varnothing$, and:
\begin{itemize}
 \item if $i\equiv 0\text{ or }1 \mod 3$ then either $\lambda^{(i)}\lessdot\lambda^{(i+1)}$ or $\lambda^{(i)} = \lambda^{(i+1)}$,
 \item if $i\equiv 2 \mod 3$ then $\lambda^{(i)}\gtrdot\lambda^{(i+1)}$.
\end{itemize}
\end{defi}

For example, 
\[
  \big( \varnothing , \;
  \yng(1) \:, \;
  \yng(2) \:; \;
  \yng(1) \:, \;
  \yng(2) \:, \;
  \yng(1,2) \:; \;
  \yng(1,1) \:, \;
  \yng(1,1) \:, \;
  \yng(1,1) \:; \;
  \yng(1) \:, \;
  \yng(1) \:, \;
  \yng(1) \:; \;
  \varnothing \big)
\]
is a stammering tableau of size 4. Note that for readability, there is a ``;'' at every three steps,
i.e., at each step going down. Through the whole article, we will have two running examples to 
illustrate the various bijections:
\[
  \Lambda_1 = 
  \big( \varnothing , \;
  \yng(1) \:, \;
  \yng(2) \:; \;
  \yng(1) \:, \;
  \yng(2) \:, \;
  \yng(1,2) \:; \;
  \yng(2) \:, \;
  \yng(2) \:, \;
  \yng(1,2) \:; \;
  \yng(1,1) \:, \;
  \yng(1,1) \:, \;
  \yng(1,1) \:; \;
  \yng(1) \:, \;
  \yng(1) \:, \;
  \yng(1) \:; \;
  \varnothing \big).
\]
and
\[
  \Lambda_2 =
  \big( \varnothing , \;
  \yng(1) \:, \;
  \yng(2) \:; \;
  \yng(1) \:, \;
  \yng(2) \:, \;
  \yng(2) \:; \;
  \yng(1) \:, \;
  \yng(1) \:, \;
  \yng(1,1) \:; \;
  \yng(1) \:, \;
  \yng(2) \:, \;
  \yng(2) \:; \;
  \yng(1) \:, \;
  \yng(1) \:, \;
  \yng(2) \:; \;
  \yng(1) \:, \;
  \yng(1) \:, \;
  \yng(1) \:; \;
  \varnothing \big).
\]

Roby \cite{roby} and Krattenthaler \cite{krattenthaler} showed that oscillating tableaux and their variants 
correspond via Fomin's growth diagrams \cite{fomin} to some fillings of Young diagrams. We can do the 
same with stammering tableaux, and we are led to the following definition:

\begin{defi}
Let $2\delta_n$ denote the {\it double staircase diagram} of $n$ rows, defined as the Young diagram with row lengths $(2n,2n-2,\dots,2)$.
A partial filling of $2\delta_n$ with some dots, such that there is exactly one dot per row and at most one per column,
is called a {\it rook placement} of size $n$.
\end{defi}

See Figure~\ref{rook} for examples. Note that it is easy to build a rook placement $T'$ of size $n$ inductively from a rook 
placement $T$ of size $n-1$. When adding a new row to the bottom of $T$, this gives {\it a priori} $2n$ possible choices for 
the location of the dot of this new row, but $n-1$ of them are forbidden by the rule that there is at most one dot per column. 
So there remains $n+1$ choice for $T'$ starting with a given $T$. In particular this shows that there are $(n+1)!$ rook 
placements of size $n$.

\Yboxdim{14pt}

\begin{figure}[h!tp]
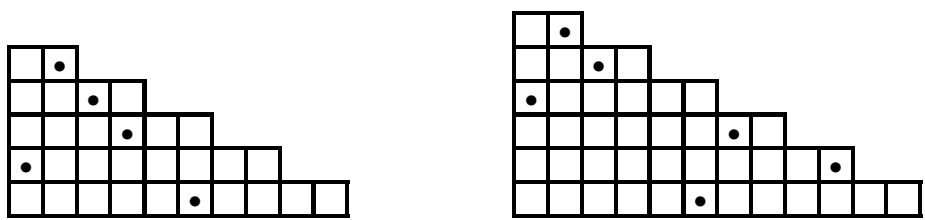
 \centering 
  \young(\n\b,\n\n\b\n,\n\n\n\b\n\n,\b\n\n\n\n\n\n\n,\n\n\n\n\n\b\n\n\n\n)         
  \hspace{2cm}
  \young(\n\b,\n\n\b\n,\b\n\n\n\n\n,\n\n\n\n\n\n\b\n,\n\n\n\n\n\n\n\n\n\b,\n\n\n\n\n\b\n\n\n\n\n\n)
\caption{Rook placements in the double staircase. \label{rook}}
\end{figure}

Let us present the description of Fomin's growth diagrams. See \cite{krattenthaler,roby}.
The idea is to assign a Young diagram to each vertex of the double staircase.
So in the pictures, we will have a big Young diagram, namely the double staircase, and 
small Young diagrams at its vertices. Note that the terminology ``growth diagram'' is rather standard
even if there is possible confusion with the notion of Young diagram.

In each cell of the double staircase, the Young diagram in the North-East corner is determined by the three 
Young diagrams in the other corners, according to the local rules given below. 
Starting with empty diagrams $\varnothing$ along the West and South border of the double staircase,
we can assign diagrams to all other vertices.
The rules are as follows. Consider the four diagrams at the corners of a cell, then, 
with the notation of Figure~\ref{cellgrowth}:
\begin{itemize}
 \item if $\mu\neq\nu$, then $\rho= \mu \cup \nu$;
 \item if $\lambda=\mu=\nu$ and the cell is empty, $\rho=\lambda$;
 \item if $\lambda=\mu=\nu$ and the cell contains a dot, then $\rho$ is obtained from $\lambda$ by adding 1 to its first part;
 \item if $\lambda\neq\mu=\nu$, then $\rho$ is obtained from $\mu$ by adding 1 to its $(k+1)$st part, where $k>0$ is minimal
       such that $\lambda_k\neq \mu_k$.
\end{itemize}

\begin{figure}[h!tp]  \centering 
 \begin{tikzpicture}
    \draw[thick] (0,0) -- (1,0) -- (1,1) -- (0,1) -- (0,0) ;  
    \node at (0.5,0.5) {$(\bullet)$};
    \node at (-0.2,-0.2) {$\lambda$};
    \node at (-0.2,1.2) {$\mu$};
    \node at (1.2,1.2) {$\rho$};
    \node at (1.2,-0.2) {$\nu$};
 \end{tikzpicture}

\caption{Generic cell of a growth diagram.  \label{cellgrowth} }
\end{figure}
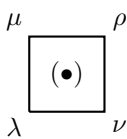

\Yboxdim{6pt}

\begin{figure}[h!tp]  \centering
  \begin{tikzpicture}[scale=0.56,thick]
    \draw[dotted,gray] (0,0) grid (2,5);
    \draw[dotted,gray] (2,0) grid (4,4);
    \draw[dotted,gray] (4,0) grid (6,3);
    \draw[dotted,gray] (6,0) grid (8,2);
    \draw[dotted,gray] (8,0) grid (10,1);
    \node at (5.5,0.5) {$\bullet$};
    \node at (0.5,1.5) {$\bullet$};
    \node at (3.5,2.5) {$\bullet$};
    \node at (2.5,3.5) {$\bullet$};
    \node at (1.5,4.5) {$\bullet$};
    \node at (0,0) {$\varnothing$};
    \node at (0,1) {$\varnothing$};    
    \node at (0,2) {$\varnothing$};    
    \node at (0,3) {$\varnothing$};    
    \node at (0,4) {$\varnothing$};    
    \node at (0,5) {$\varnothing$};    
    \node at (1,0) {$\varnothing$};    
    \node at (1,1) {$\varnothing$};    
    \node at (1,2) {\yng(1)};    
    \node at (1,3) {\yng(1)};    
    \node at (1,4) {\yng(1)};    
    \node at (1,5) {\yng(1)};    
    \node at (2,0) {$\varnothing$};    
    \node at (2,1) {$\varnothing$};    
    \node at (2,2) {\yng(1)};
    \node at (2,3) {\yng(1)};
    \node at (2,4) {\yng(1)};
    \node at (2,5) {\yng(2)};
    \node at (3,0) {$\varnothing$};    
    \node at (3,1) {$\varnothing$};    
    \node at (3,2) {\yng(1)}; 
    \node at (3,3) {\yng(1)};
    \node at (3,4) {\yng(2)};
    \node at (4,0) {$\varnothing$};    
    \node at (4,1) {$\varnothing$};    
    \node at (4,2) {\yng(1)};
    \node at (4,3) {\yng(2)};
    \node at (4,4) {\yng(1,2)};
    \node at (5,0) {$\varnothing$};    
    \node at (5,1) {$\varnothing$};    
    \node at (5,2) {\yng(1)};
    \node at (5,3) {\yng(2)};
    \node at (6,0) {$\varnothing$};    
    \node at (6,1) {\yng(1)};
    \node at (6,2) {\yng(1,1)};
    \node at (6,3) {\yng(1,2)};
    \node at (7,0) {$\varnothing$};    
    \node at (7,1) {\yng(1)};
    \node at (7,2) {\yng(1,1)};
    \node at (8,0) {$\varnothing$};    
    \node at (8,1) {\yng(1)};
    \node at (8,2) {\yng(1,1)};
    \node at (9,0) {$\varnothing$};    
    \node at (9,1) {\yng(1)};
    \node at (10,0) {$\varnothing$};    
    \node at (10,1) {\yng(1)};
  \end{tikzpicture}
  \hspace{1cm}
  \begin{tikzpicture}[scale=0.56,thick]
    \draw[dotted,gray] (0,0) grid (2,6);
    \draw[dotted,gray] (2,0) grid (4,5);
    \draw[dotted,gray] (4,0) grid (6,4);
    \draw[dotted,gray] (6,0) grid (8,3);
    \draw[dotted,gray] (8,0) grid (10,2);   
    \draw[dotted,gray] (10,0) grid (12,1);   
    \node at (1.5,5.5) {$\bullet$};
    \node at (2.5,4.5) {$\bullet$};
    \node at (0.5,3.5) {$\bullet$};
    \node at (6.5,2.5) {$\bullet$};
    \node at (9.5,1.5) {$\bullet$};
    \node at (5.5,0.5) {$\bullet$};    
    \node at (0,0) {$\varnothing$};
    \node at (0,1) {$\varnothing$};    
    \node at (0,2) {$\varnothing$};    
    \node at (0,3) {$\varnothing$};    
    \node at (0,4) {$\varnothing$};    
    \node at (0,5) {$\varnothing$};    
    \node at (0,6) {$\varnothing$};    
    \node at (1,0) {$\varnothing$};    
    \node at (1,1) {$\varnothing$};    
    \node at (1,2) {$\varnothing$};    
    \node at (1,3) {$\varnothing$};    
    \node at (1,4) {\yng(1)};    
    \node at (1,5) {\yng(1)};    
    \node at (1,6) {\yng(1)};    
    \node at (2,0) {$\varnothing$};    
    \node at (2,1) {$\varnothing$};    
    \node at (2,2) {$\varnothing$};    
    \node at (2,3) {$\varnothing$};    
    \node at (2,4) {\yng(1)};
    \node at (2,5) {\yng(1)};
    \node at (2,6) {\yng(2)};
    \node at (3,0) {$\varnothing$};    
    \node at (3,1) {$\varnothing$};    
    \node at (3,2) {$\varnothing$};   
    \node at (3,3) {$\varnothing$};   
    \node at (3,4) {\yng(1)};
    \node at (3,5) {\yng(2)};
    \node at (4,0) {$\varnothing$};    
    \node at (4,1) {$\varnothing$};    
    \node at (4,2) {$\varnothing$};    
    \node at (4,3) {$\varnothing$};    
    \node at (4,4) {\yng(1)};
    \node at (4,5) {\yng(2)};
    \node at (5,0) {$\varnothing$};    
    \node at (5,1) {$\varnothing$};    
    \node at (5,2) {$\varnothing$};    
    \node at (5,3) {$\varnothing$};    
    \node at (5,4) {\yng(1)};
    \node at (6,0) {$\varnothing$};    
    \node at (6,1) {\yng(1)};
    \node at (6,2) {\yng(1)};
    \node at (6,3) {\yng(1)};
    \node at (6,4) {\yng(1,1)};
    \node at (7,0) {$\varnothing$};    
    \node at (7,1) {\yng(1)};
    \node at (7,2) {\yng(1)};
    \node at (7,3) {\yng(2)};
    \node at (8,0) {$\varnothing$};    
    \node at (8,1) {\yng(1)};
    \node at (8,2) {\yng(1)};
    \node at (8,3) {\yng(2)};
    \node at (9,0) {$\varnothing$};    
    \node at (9,1) {\yng(1)};
    \node at (9,2) {\yng(1)};
    \node at (10,0) {$\varnothing$};    
    \node at (10,1) {\yng(1)};
    \node at (10,2) {\yng(2)};
    \node at (11,0) {$\varnothing$};    
    \node at (11,1) {\yng(1)};
    \node at (12,0) {$\varnothing$};    
    \node at (12,1) {\yng(1)};
  \end{tikzpicture}
  \caption{Growth diagrams. \label{growth}}
\end{figure}

The growth diagrams obtained from the rook placements in Figure~\ref{rook} are in Figure~\ref{growth}.
By following the North-East boundary, from the top-left corner to the bottom-right corner, 
we read a sequence of Young diagrams. The two growth diagrams in Figure~\ref{growth} give our two
running examples $\Lambda_1$ and $\Lambda_2$ defined at the beginning of this section.

The same bijection can be described using the Schensted insertion instead of the growth diagrams. 
This algorithm is described in \cite[Chapter~7]{stanley2}.
An {\it incomplete standard tableau} is a filling of a Young diagram with integers such that:
\begin{itemize}
 \item the integers are increasing along rows (from left to right) and columns (from bottom to top),
 \item there are no repetition among them.
\end{itemize}
Let $T$ be such a tableau.
If $k>0$ is not an entry of $T$, let $T\leftarrow k$ denote the tableau obtained after row insertion of $k$
(see \cite[Chapter~7]{stanley2}). 
If $k$ is an entry of $T$ located at a corner, let $T\rightarrow k$ denote the tableau obtained after removing $k$.

From a rook placement $R$ of size $n$, we can form a sequence $T_0,\dots,T_{3n}$ of incomplete standard tableaux in the 
following way. We label the dots in $R$ so that the dot in the $k$th row (from bottom to top) has label $k$, see Figure~\ref{rook2}.
We read the North-East border of $R$ from the top-left corner to the bottom-right corner. This border 
contains $3n$ steps, and starting from the empty tableau $T_0$, we define $T_0,\dots,T_{3n}$ in such a way that:
\begin{itemize}
 \item if the $i$th step is horizontal, at the top of a column which contains no dot, then $T_{i}=T_{i-1}$,
 \item if the $i$th step is horizontal, at the top of a column which contains the dot with label $k$, then $T_{i}= T_{i-1} \leftarrow k $,
 \item if the $i$th step is vertical, at the right of a row which contains the dot with label $k$, then $T_{i}=T_{i-1} \rightarrow k$.
\end{itemize}
For example, the rook placements in Figure~\ref{rook2} respectively give the sequences

\scriptsize

\Yboxdim{12pt}
\[
  \varnothing, 
  \young(2) 
  \:, \;
  \young(25) 
  \:; \;
  \young(2) 
  \:, \;
  \young(24) 
  \:, \;
  \young(4,23) 
  \:; \;
  \young(23) 
  \:, \;
  \young(23) 
  \:, \;
  \young(2,13) 
  \:; \;
  \young(2,1) 
  \:, \;
  \young(2,1) 
  \:, \;
  \young(2,1) 
  \:; \;
  \young(1) 
  \:, \;
  \young(1) 
  \:, \;
  \young(1) 
  \:; \; 
  \varnothing
\]
\normalsize
and
\scriptsize
\[
  \varnothing , \;
  \young(4) \:, \;
  \young(46) \:; \;
  \young(4) \:, \;
  \young(45) \:, \;
  \young(45) \:; \;
  \young(4) \:, \;
  \young(4) \:, \;
  \young(4,1) \:; \;
  \young(1) \:, \;
  \young(13) \:, \;
  \young(13) \:; \;
  \young(1) \:, \;
  \young(1) \:, \;
  \young(12) \:; \;
  \young(1) \:, \;
  \young(1) \:, \;
  \young(1) \:; \;  
  \varnothing.
\]
\normalsize
If we forget the integer entries, the sequence of shapes of the incomplete standard tableaux $T_0,\dots,T_{3n}$ form a 
stammering tableau $\lambda^{(0)},\dots, \lambda^{(3n)}$, and it is precisely the one associated to $R$ via the growth diagram.
This is a general fact and follows from known properties of these constructions \cite{fomin,krattenthaler,roby}.

\Yboxdim{14pt}

\begin{figure}[h!tp]
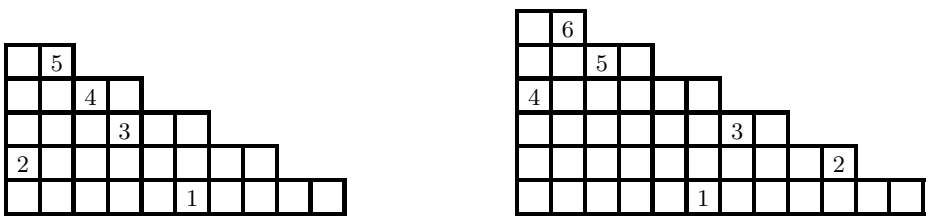
    \centering 
  {\small \young(\n 5,\n\n 4\n,\n\n\n 3\n\n,2\n\n\n\n\n\n\n,\n\n\n\n\n 1\n\n\n\n) }
  \hspace{2cm}
  {\small \young(\n 6,\n\n 5\n,4\n\n\n\n\n,\n\n\n\n\n\n 3\n,\n\n\n\n\n\n\n\n\n 2,\n\n\n\n\n 1\n\n\n\n\n\n) }
\caption{Rook placements in double staircase, with labeled dots. \label{rook2}}
\end{figure}

Finally, let us present a third equivalent construction via Viennot's shadow lines \cite{viennot2}.
We give only one example of this construction, in Figure~\ref{shadows} (page~\pageref{shadows}).

Start from a rook placement $R$ in the double staircase. The idea is that each dot
projects a quarterplane-like shadow to its North-East. 
The boundary of the union of all these shadows gives a curve, the {\it first shadow}.
We can do the same thing with the dots left alone by the first shadow, this gives the second shadow.
The other shadows are constructed in the same way. In the first picture of the example, we see that
only two shadows are necessary to cover all dots of the rook placement. Now, the idea is that 
these shadow lines determine the first rows of the partitions in a stammering tableau as follows.
Starting from 0, we read the $3n$ steps of the North-East boundary of the diagram, add 1 when reading
a horizontal step crossing a shadow line, remove 1 when reading a vertical step crossing a shadow line,
doing nothing otherwise.

The next step is to put a dot at each $\daleth$-shaped corner of the shadows.
We obtain a new set of dots, which has to be a rook placement as well.
The shadows of these dots gives the second rows in the stammering tableau.
See the second picture in the example of Figure~\ref{shadows}.

Going on with the successive shadows of the successive new sets of dots give the full stammering tableau.
We don't give the proof that this construction is equivalent to the previous ones, and refer to \cite{viennot2} 
for more details.

\begin{rema}
 We have introduced stammering tableaux using the interpretation of the operators $D$ and $U$ acting on
 Young diagrams, and given bijections with the rook placements in the double staircase using various methods.
 It is also possible to obtain the rook placements from the operators $D$ and $U$ (as abstract operators
 satisfying a commutation relation $DU-UD=I$) using the methods in \cite{varvak}.
\end{rema}

\begin{rema}
 Although we do not give the details here, it is possible to give {\it reverse local rules}
 for growth diagrams. See \cite[Section~2.1.8]{langer}.
 In the notation of Figure~\ref{growth}, it means that if we know $\mu$, $\nu$,
 and $\rho$, we can recover $\lambda$ and whether the cell contains a dot or not. So if we have a 
 stammering tableaux along the North-East border of the double staircase diagram, we can apply these
 rules to complete the rest of the picture, and get a rook placement. This is the reverse bijection of
 the one presented in this section.
\end{rema}

\begin{figure}[h!tp]  \centering 
  \begin{tikzpicture}[scale=0.6,thick]
    \draw (0,0) grid (2,7);
    \draw (2,0) grid (4,6);
    \draw (4,0) grid (6,5);
    \draw (6,0) grid (8,4);
    \draw (8,0) grid (10,3);
    \draw (10,0) grid (12,2);
    \draw (12,0) grid (14,1);    
    \draw[red,line width=2pt] (0.5,7.3) -- (0.5,6.5) -- (1.5,6.5) -- (1.5,3.5) -- (5.5,3.5) -- (5.5,2.5) -- (6.5,2.5) -- (6.5,0.5) -- (14.3,0.5);
    \draw[red,line width=2pt] (2.5,6.3) -- (2.5,5.5) -- (3.5,5.5) -- (3.5,4.5) -- (6.3,4.5);
    \draw[red,line width=2pt] (8.5,3.3) -- (8.5,1.5) -- (12.3,1.5);
    \node at (6.5,0.5) {$\bullet$};
    \node at (8.5,1.5) {$\bullet$};
    \node at (5.5,2.5) {$\bullet$};
    \node at (1.5,3.5) {$\bullet$};
    \node at (3.5,4.5) {$\bullet$};
    \node at (2.5,5.5) {$\bullet$};
    \node at (0.5,6.5) {$\bullet$};
  \end{tikzpicture}
  
  \vspace{7mm}
  
  \begin{tikzpicture}[scale=0.6,thick]
    \draw (0,0) grid (2,7);
    \draw (2,0) grid (4,6);
    \draw (4,0) grid (6,5);
    \draw (6,0) grid (8,4);
    \draw (8,0) grid (10,3);
    \draw (10,0) grid (12,2);
    \draw (12,0) grid (14,1);    
    \draw[blue,line width=2pt] (1.5,7.3) -- (1.5,6.5) -- (2.3,6.5);
    \draw[blue,line width=2pt] (3.5,6.3) -- (3.5,5.5) -- (4.3,5.5);
    \draw[blue,line width=2pt] (5.5,5.3) -- (5.5,3.5) -- (6.5,3.5) -- (6.5,2.5) -- (10.3,2.5);
    \node at (6.5,2.5) {$\bullet$};
    \node at (5.5,3.5) {$\bullet$};
    \node at (3.5,5.5) {$\bullet$};
    \node at (1.5,6.5) {$\bullet$};
  \end{tikzpicture}
  
  \vspace{7mm}
  
  \begin{tikzpicture}[scale=0.6,thick]
    \draw (0,0) grid (2,7);
    \draw (2,0) grid (4,6);
    \draw (4,0) grid (6,5);
    \draw (6,0) grid (8,4);
    \draw (8,0) grid (10,3);
    \draw (10,0) grid (12,2);
    \draw (12,0) grid (14,1);    
    \draw[green,line width=2pt] (6.5,4.3) -- (6.5,3.5) -- (8.3,3.5);
    \node at (6.5,3.5) {$\bullet$};
  \end{tikzpicture}
  
  \vspace{8mm}

  $\varnothing$, 
  \begin{tikzpicture}[scale=0.3,thick]
    \draw[fill=red] (0,0) -- (0,1) -- (1,1) -- (1,0) -- (0,0);
  \end{tikzpicture} , 
  \begin{tikzpicture}[scale=0.3,thick]
    \draw[fill=blue] (0,2) -- (0,1) -- (1,1) -- (1,2) -- (0,2);
    \draw[fill=red] (0,0) -- (0,1) -- (1,1) -- (1,0) -- (0,0);
  \end{tikzpicture} ; 
  \begin{tikzpicture}[scale=0.3,thick]
    \draw[fill=red] (0,0) -- (0,1) -- (1,1) -- (1,0) -- (0,0);
  \end{tikzpicture} , 
  \begin{tikzpicture}[scale=0.3,thick]
    \draw[fill=red] (0,0) -- (0,1) -- (2,1) -- (2,0) -- (0,0);
    \draw (1,0) -- (1,1);
  \end{tikzpicture} , 
  \begin{tikzpicture}[scale=0.3,thick]
    \draw[fill=blue] (0,2) -- (0,1) -- (1,1) -- (1,2) -- (0,2);
    \draw[fill=red] (0,0) -- (0,1) -- (2,1) -- (2,0) -- (0,0);
    \draw (1,0) -- (1,1);
  \end{tikzpicture} ;
  \begin{tikzpicture}[scale=0.3,thick]
    \draw[fill=red] (0,0) -- (0,1) -- (2,1) -- (2,0) -- (0,0);
    \draw (1,0) -- (1,1);
  \end{tikzpicture} , 
  \begin{tikzpicture}[scale=0.3,thick]
    \draw[fill=red] (0,0) -- (0,1) -- (2,1) -- (2,0) -- (0,0);
    \draw (1,0) -- (1,1);
  \end{tikzpicture} , 
  \begin{tikzpicture}[scale=0.3,thick]
    \draw[fill=blue] (0,2) -- (0,1) -- (1,1) -- (1,2) -- (0,2);
    \draw[fill=red] (0,0) -- (0,1) -- (2,1) -- (2,0) -- (0,0);
    \draw (1,0) -- (1,1);
  \end{tikzpicture} ;
  \begin{tikzpicture}[scale=0.3,thick]
    \draw[fill=blue] (0,2) -- (0,1) -- (1,1) -- (1,2) -- (0,2);
    \draw[fill=red] (0,0) -- (0,1) -- (1,1) -- (1,0) -- (0,0);
  \end{tikzpicture} ,
  \begin{tikzpicture}[scale=0.3,thick]
    \draw[fill=green] (0,3) -- (0,2) -- (1,2) -- (1,3) -- (0,3);
    \draw[fill=blue] (0,2) -- (0,1) -- (1,1) -- (1,2) -- (0,2);
    \draw[fill=red] (0,0) -- (0,1) -- (1,1) -- (1,0) -- (0,0);
  \end{tikzpicture} ,
  \begin{tikzpicture}[scale=0.3,thick]
    \draw[fill=green] (0,3) -- (0,2) -- (1,2) -- (1,3) -- (0,3);
    \draw[fill=blue] (0,2) -- (0,1) -- (1,1) -- (1,2) -- (0,2);
    \draw[fill=red] (0,0) -- (0,1) -- (1,1) -- (1,0) -- (0,0);
  \end{tikzpicture} ;
  
  \vspace{4mm}
  
  \begin{tikzpicture}[scale=0.3,thick]
    \draw[fill=blue] (0,2) -- (0,1) -- (1,1) -- (1,2) -- (0,2);
    \draw[fill=red] (0,0) -- (0,1) -- (1,1) -- (1,0) -- (0,0);
  \end{tikzpicture} ,
  \begin{tikzpicture}[scale=0.3,thick]
    \draw[fill=blue] (0,2) -- (0,1) -- (1,1) -- (1,2) -- (0,2);
    \draw[fill=red] (0,0) -- (0,1) -- (2,1) -- (2,0) -- (0,0);
    \draw (1,0) -- (1,1);
  \end{tikzpicture} ,
  \begin{tikzpicture}[scale=0.3,thick]
    \draw[fill=blue] (0,2) -- (0,1) -- (1,1) -- (1,2) -- (0,2);
    \draw[fill=red] (0,0) -- (0,1) -- (2,1) -- (2,0) -- (0,0);
    \draw (1,0) -- (1,1);
  \end{tikzpicture} ;
  \begin{tikzpicture}[scale=0.3,thick]
    \draw[fill=red] (0,0) -- (0,1) -- (2,1) -- (2,0) -- (0,0);
    \draw (1,0) -- (1,1);
  \end{tikzpicture} ,
  \begin{tikzpicture}[scale=0.3,thick]
    \draw[fill=red] (0,0) -- (0,1) -- (2,1) -- (2,0) -- (0,0);
    \draw (1,0) -- (1,1);
  \end{tikzpicture} ,
  \begin{tikzpicture}[scale=0.3,thick]
    \draw[fill=red] (0,0) -- (0,1) -- (2,1) -- (2,0) -- (0,0);
    \draw (1,0) -- (1,1);
  \end{tikzpicture} ;
  \begin{tikzpicture}[scale=0.3,thick]
    \draw[fill=red] (0,0) -- (0,1) -- (1,1) -- (1,0) -- (0,0);
  \end{tikzpicture} ,
  \begin{tikzpicture}[scale=0.3,thick]
    \draw[fill=red] (0,0) -- (0,1) -- (1,1) -- (1,0) -- (0,0);
  \end{tikzpicture} ,
  \begin{tikzpicture}[scale=0.3,thick]
    \draw[fill=red] (0,0) -- (0,1) -- (1,1) -- (1,0) -- (0,0);
  \end{tikzpicture} ;
  $\varnothing$
    
  \caption{Shadows of a rook placement. \label{shadows} }
\end{figure}

\section{Chains of Dyck shapes}

\label{sec:3}

A {\it Dyck path of length} $2n$ is a path in $\mathbb{N}^2$ from $(0,0)$ to $(2n,0)$ with steps $(1,1)$ and $(1,-1)$.
The two kind of steps will be denoted $\nearrow = (1,1)$ and $\searrow=(1,-1)$.
The {\it height} of $(x,y)\in\mathbb{N}^2$ is its second coordinate $y$. Accordingly, each step in a Dyck path
has an {\it initial height} and a {\it final height}, namely the heights of its starting point and ending point, respectively.

We can also see a Dyck path of length $2n$ as a binary word of length $2n$ over $\{\nearrow,\searrow\}$ such that
in each prefix, there are more $\nearrow$ than $\searrow$ (in the weak sense). These words are called {\it Dyck words}.
A prefix of a Dyck word is called a {\it Dyck prefix}.

A property of stammering tableaux or rook placements that will be important in the sequel is that we can naturally 
associate a Dyck path to each of them.

\begin{defi}
 Let $R$ be a rook placement in the double staircase $2\delta_n$. We define a Dyck path $d(R)$ of length $2n+2$ as follows:
 \begin{itemize}
  \item the first step is $\nearrow$, the last step is $\searrow$,
  \item if $2\leq i\leq 2n+1$, the $i$th step is $\nearrow$ if the $(i-1)$st column of $R$ contains a dot and $\searrow$ otherwise.
 \end{itemize}
\end{defi}

See Figure~\ref{Rdyck} for an example.

\begin{figure}[h!tp]   \centering
  \begin{tikzpicture}[scale=0.4,thick]
    \draw[dotted,gray] (0,0) grid (12,6);
    \draw (0,0) -- (5,5) -- (6,4) -- (7,5) -- (12,0);
  \end{tikzpicture}
 \hspace{2cm}
  \begin{tikzpicture}[scale=0.4,thick]
    \draw[dotted,gray] (0,0) grid (14,6);
    \draw (0,0) -- (4,4) -- (6,2) -- (8,4) -- (10,2) -- (11,3) -- (14,0);
  \end{tikzpicture}
 \caption{ The Dyck paths associated to the rook placements in Figure~\ref{rook}. \label{Rdyck} }
\end{figure}
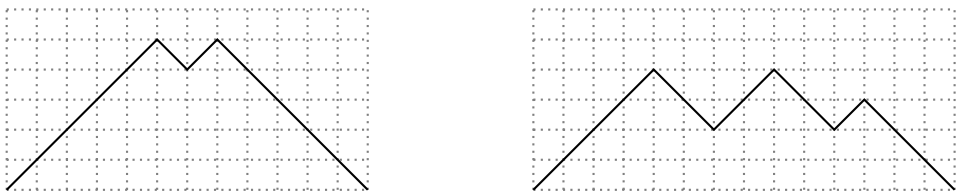

\begin{lemm}
 With the previous definition, $d(R)$ is indeed a Dyck path.
\end{lemm}

\begin{proof}
 Let $1\leq i \leq 2n$. Among the first $i$ columns of $R$, say that $x$ contains a dot and $y$ are empty (so $x+y=i$).
 In the double staircase, the first $i$ columns entirely contain the topmost $\lfloor i/2 \rfloor$ rows, 
 so $x \geq \lfloor i/2 \rfloor$, and $y=i-x\leq \lceil i/2 \rceil$. The first $i+1$ steps of $d(R)$ contains
 $1+x$ steps $\nearrow$ and $y$ steps $\searrow$, so the height after the $(i+1)$st step is $1+x-y$.
 We have $1+x-y\geq 1+ \lfloor i/2 \rfloor - \lceil i/2 \rceil \geq 0$ so that the path $d(R)$ 
 stays above the horizontal axis.
\end{proof}

Instead of Dyck paths, it will convenient to use the equivalent notion given below.

\begin{defi}
 The {\it staircase partition} of size $n$ is $\delta_n=(n,n-1,\dots,1) \in \mathcal{Y}$.
 A {\it Dyck shape} of length $2n$ is a skew shape $\delta_n / \lambda$ where $\lambda\in\mathcal{Y}$ is such that $\lambda\subset\delta_{n-1}$.
\end{defi}

We will use here an unusual notation for Young diagrams, that we call {\it Japanese notation}. It is obtained from the French notation by performing 
a 135 degree clockwise rotation.
By following the upper border of a Dyck shape in Japanese notation, we recover a Dyck path. 
See Figure~\ref{dyck} for an example with $n=5$ and $\lambda=(3,1,1,1)$. 
Note that the parts of the partition correspond to the number of cells in each South-East to North-West diagonal.

\begin{figure}[h!tp]   \centering
  \begin{tikzpicture}[scale=0.4,thick]
    \draw[dotted,gray] (0,0) grid (10,4);
    \draw (0,0) -- (2,2) -- (3,1) -- (5,3) -- (8,0) -- (9,1) -- (10,0);
  \end{tikzpicture}
 \hspace{2cm}
   \begin{tikzpicture}[scale=0.4,thick]
    \draw[dotted,gray] (2,2) -- (5,5) -- (9,1);
    \draw[dotted,gray] (3,3) -- (4,2);
    \draw[dotted,gray] (4,4) -- (5,3) -- (6,4);
    \draw[dotted,gray] (6,2) -- (7,3);
    \draw[dotted,gray] (7,1) -- (8,2);
    \draw (0,0) -- (2,2) -- (5,-1) -- (7,1) -- (9,-1) -- (10,0) -- (9,1) -- (7,-1) -- (4,2) -- (1,-1) -- (0,0);
    \draw (1,1) -- (3,-1) -- (6,2) -- (7,1);
    \draw (6,2) -- (5,3) -- (4,2);
   \end{tikzpicture}
 \caption{A Dyck shape. \label{dyck} }
\end{figure}
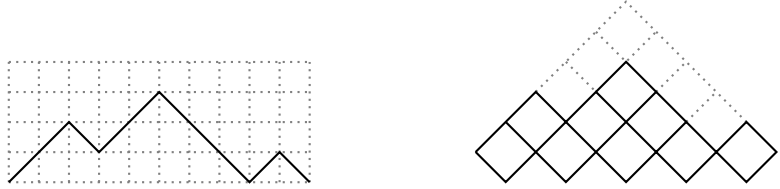

Each (unit) square cell of a Dyck shape has some coordinates $(x,y)$. They are normalized so that the bottom left cell has 
coordinates $(0,0)$, the one to its North-East has coordinates $(1,1)$, the bottom right cell has coordinates $(2n-2,0)$, and 
so on.
A {\it column} of a Dyck shape is a set of cells having the same $x$-coordinate, for example the Dyck shape in Figure~\ref{dyck} has 8 non-empty
columns, one of them having 2 cells and the others only 1 cell. Similarly, a {\it row} of a Dyck shape is a set of cells having the same 
$y$-coordinate. In the example of Figure~\ref{dyck}, there are three non-empty rows. The notion of diagonals should be clear,
to distinguish the two different kinds, we refer to them as $\diagup$-diagonals and $\diagdown$-diagonals.

\begin{defi}
 A skew shape is called a {\it ribbon} if it is connected and contains no $2\times 2$ square.
 Let $D$ and $E$ be two Dyck shapes of respective length $2n$ and $2n+2$, 
 then we denote $D\sqsubset E$ and say that $E$ is obtained from $D$ by {\it addition of a ribbon} if
 $D\subset E$ and the difference $E/D$ is a ribbon.
\end{defi}

In this definition, we assume that the two Dyck shapes are arranged so that their leftmost 
cells coincide (which is coherent with the way we defined the coordinates in the skew shape).

\begin{rema}
 This definition is easily translated in terms of binary words over $\nearrow$ and $\searrow$. Let
 $D$ a dyck word, then $D\sqsubset E$ if and only if $E$ is obtained from $D\searrow\searrow$ 
 by changing a $\searrow$ into a $\nearrow$ (and each step $\searrow$ of $D\searrow\searrow$ can
 be changed except the last one, so that there are $n+1$ possibilities for a path of length $2n$).
\end{rema}

Figure~\ref{addribbon} shows the ribbons that can be added to the Dyck path of Figure~\ref{dyck}. 
Note that the number of $\diagup$-diagonals of these ribbons are exactly the integers from 1 to 6. 
This is a general fact:

\begin{prop} \label{addribbonprop}
Let $D$ be a Dyck shape of length $2n$ and $1\leq i\leq n+1$. Then there is a unique way to add to $D$ 
a ribbon whose number of $\diagup$-diagonals is $i$.
\end{prop}

The result is clear upon inspection and does not deserve a detailed formal proof.

\begin{figure}[h!tp]   \centering 
  \begin{tikzpicture}[scale=0.3,thick]
    \draw[fill=lightgray] (12,0) -- (11,1) -- (10,0) -- (11,-1) -- (12,0);
    \draw (0,0) -- (2,2) -- (5,-1) -- (7,1) -- (9,-1) -- (11,1) -- (12,0) -- (11,-1) -- (9,1) -- (7,-1) -- (4,2) -- (1,-1) -- (0,0);
    \draw (1,1) -- (3,-1) -- (6,2); 
    \draw (4,2) -- (5,3) -- (7,1);
  \end{tikzpicture}
  \hspace{1cm}
  \begin{tikzpicture}[scale=0.3,thick]
    \draw[fill=lightgray] (12,0) -- (10,2) -- (9,1) -- (11,-1) -- (12,0);
    \draw (0,0) -- (2,2) -- (5,-1) -- (7,1) -- (9,-1) -- (11,1) -- (12,0) -- (11,-1) -- (9,1) -- (7,-1) -- (4,2) -- (1,-1) -- (0,0);
    \draw (1,1) -- (3,-1) -- (6,2); 
    \draw (4,2) -- (5,3) -- (7,1);
    \draw (11,1) -- (10,2) -- (9,1);
  \end{tikzpicture}
  \hspace{1cm}
  \begin{tikzpicture}[scale=0.3,thick]
    \draw[fill=lightgray] (12,0) -- (10,2) -- (9,1) -- (11,-1) -- (12,0);
    \draw[fill=lightgray] (7,1) -- (8,0) -- (10,2) -- (9,3) -- (7,1);
    \draw (0,0) -- (2,2) -- (5,-1) -- (7,1) -- (9,-1) -- (11,1) -- (12,0) -- (11,-1) -- (9,1) -- (7,-1) -- (4,2) -- (1,-1) -- (0,0);
    \draw (1,1) -- (3,-1) -- (6,2); 
    \draw (4,2) -- (5,3) -- (7,1);
    \draw (11,1) -- (10,2) -- (9,1);
    \draw (9,1) -- (8,2);
  \end{tikzpicture}
  
  \vspace{4mm}
  
  \begin{tikzpicture}[scale=0.3,thick]
    \draw[fill=lightgray] (12,0) -- (9,3) -- (8,2) -- (11,-1) -- (12,0);
    \draw[fill=lightgray] (6,2) -- (7,3) -- (9,1) -- (8,0) -- (6,2);
    \draw (0,0) -- (2,2) -- (5,-1) -- (7,1) -- (9,-1) -- (11,1) -- (12,0) -- (11,-1) -- (9,1) -- (7,-1) -- (4,2) -- (1,-1) -- (0,0);
    \draw (1,1) -- (3,-1) -- (6,2); 
    \draw (4,2) -- (5,3) -- (7,1);
    \draw (11,1) -- (10,2) -- (9,1);
    \draw (9,1) -- (8,2);
    \draw (7,1) -- (8,2);
  \end{tikzpicture}
  \hspace{1cm}
  \begin{tikzpicture}[scale=0.3,thick]
    \draw[fill=lightgray] (12,0) -- (9,3) -- (8,2) -- (11,-1) -- (12,0);
    \draw[fill=lightgray] (5,3) -- (6,4) -- (9,1) -- (8,0) -- (5,3);
    \draw (0,0) -- (2,2) -- (5,-1) -- (7,1) -- (9,-1) -- (11,1) -- (12,0) -- (11,-1) -- (9,1) -- (7,-1) -- (4,2) -- (1,-1) -- (0,0);
    \draw (1,1) -- (3,-1) -- (6,2); 
    \draw (4,2) -- (5,3) -- (7,1);
    \draw (11,1) -- (10,2) -- (9,1);
    \draw (9,1) -- (8,2);
    \draw (7,1) -- (8,2);
    \draw (7,3) -- (6,2);
  \end{tikzpicture}
  \hspace{1cm}
  \begin{tikzpicture}[scale=0.3,thick]
    \draw[fill=lightgray] (12,0) -- (9,3) -- (8,2) -- (11,-1) -- (12,0);
    \draw[fill=lightgray] (5,3) -- (6,4) -- (9,1) -- (8,0) -- (5,3);
    \draw[fill=lightgray] (6,4) -- (5,5) -- (2,2) -- (3,1) -- (6,4);
    \draw (0,0) -- (2,2) -- (5,-1) -- (7,1) -- (9,-1) -- (11,1) -- (12,0) -- (11,-1) -- (9,1) -- (7,-1) -- (4,2) -- (1,-1) -- (0,0);
    \draw (1,1) -- (3,-1) -- (6,2); 
    \draw (4,2) -- (5,3) -- (7,1);
    \draw (11,1) -- (10,2) -- (9,1);
    \draw (9,1) -- (8,2);
    \draw (7,1) -- (8,2);
    \draw (7,3) -- (6,2);
    \draw (3,3) -- (4,2);
    \draw (4,4) -- (5,3);
  \end{tikzpicture}
  
  \caption{Addition of a ribbon to a Dyck path. \label{addribbon} }
\end{figure}
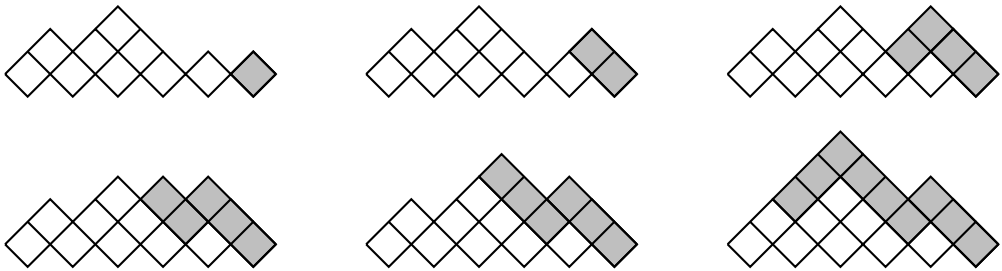

\begin{defi}
  An {\it $n$-chain of Dyck shapes} is a sequence $D_1\sqsubset D_2 \sqsubset \dots \sqsubset D_n$ 
  where $D_i$ is a Dyck shape of length $2i$.
  The biggest path $D_n$ is called the {\it shape} of the chain, and we say the chains ends at $D_n$.
\end{defi}

Note that from Proposition~\ref{addribbonprop}, the number of $n$-chains of Dyck shapes is $n!$.

When we have an $n$-chain $D_1\sqsubset D_2 \sqsubset \dots \sqsubset D_n$, the ribbons $D_{i}/D_{i-1}$ with $1\leq i \leq n$ 
(where we take the convention $D_0=\varnothing$) form a partition of $D_n$ as a set of cells.
We can thus represent an $n$-chain ending at $D$ as a partition of $D$ in $n$ ribbons, see Figure~\ref{chain} for an example.
Conversely, a Dyck shape $D$ of length $2n$ partitioned in $n$ ribbons arise in this way if it satisfies the following condition:
the right extremities of the ribbons are exactly the $n$ cells in the bottom row of $D$.

\begin{prop}
  Let $R$ be a rook placement in $2\delta_{n-1}$. Let $R_i$ be the rook placement in $2\delta_i$ obtained by keeping only 
  the $i$ top rows of $R$ (by convention, $R_0$ is empty so that $d(R_0)= \delta_1 / \varnothing $, the unique
  Dyck shape of length 2, and $R_{n-1}=R$). 
  Then $d(R_0) , \dots , d(R_{n-1})$ is an $n$-chain of Dyck shapes. This defines a bijection between 
  rook placements in $2\delta_{n-1}$ and $n$-chains of Dyck shapes.
\end{prop}

\begin{proof}
Let us compare $d(R_{i-1})$ and $d(R_i)$. If the dot in the $i$th row of $R_i$ is on the right extremity, 
$d(R_i)$ is obtained from $d(R_{i-1})$ by adding $\nearrow\searrow$ at the end. In the other cases,
since only one dot is added from $R_{i-1}$ to $R_i$, $d(R_i)$ as a Dyck word is obtained from $d(R_{i-1})$ by 
replacing a $\searrow$ with a $\nearrow$ and adding two $\searrow$
at the end. This shows $d(R_{i-1})\sqsubset d(R_i)$, so $d(R_0) , \dots , d(R_{n-1})$ is indeed a chain of Dyck shapes.

We can recover the rook placement $R$ from the chain of Dyck shapes. Indeed, let $j$ be the index of the first
step where the two paths $d(R_{i-1})$ and $d(R_i)$ differs, or $j=2i+1$ if $d(R_{i-1})$ is a prefix of $d(R_i)$.
Then there is a dot in the $(j-1)$st cell of the $i$th row of $R$.

Since both sets have the same cardinality, what we have defined is a bijection.
\end{proof}

\begin{figure}[h!tp]   \centering 
  \begin{tikzpicture}[scale=0.3,thick]
    \draw (6,4) -- (7,5) -- (12,0) -- (11,-1) -- (5,5) -- (0,0) -- (1,-1) -- (5,3) -- (9,-1) -- (10,0);
    \draw (1,1) -- (3,-1) -- (4,0);
    \draw (3,1) -- (5,-1) -- (6,0);
    \draw (4,2) -- (7,-1) -- (8,0);
  \end{tikzpicture}
  \hspace{2cm}
  \begin{tikzpicture}[scale=0.3,thick]
   \draw (8,0)-- (7,-1) -- (4,2) -- (1,-1) -- (0,0) -- (4,4) -- (9,-1) -- (11,1) -- (13,-1) -- (14,0) -- (11,3) -- (10,2) -- (8,4) -- (6,2);
   \draw (1,1) -- (3,-1) -- (4,0);
   \draw (3,1) -- (5,-1) -- (6,0);
   \draw (7,1) -- (8,2) -- (11,-1) -- (12,0);
  \end{tikzpicture}
  \caption{The chains of Dyck shapes corresponding to rook placements in Figure~\ref{rook}. \label{chain}}
\end{figure}
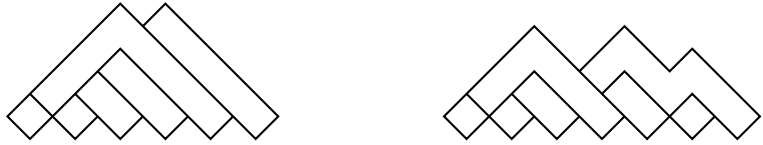

\section{The poset of Dyck paths}

\label{sec:4}

The transitive and reflexive closure of the relation $\sqsubset$ defines the structure of a graded poset on the set of 
Dyck paths of all possible lengths. The rank of a path is half its length, for example the minimal
Dyck path (i.e., the empty path) has rank 0, and $\nearrow\searrow$ has rank 1.
In terms of this poset, what we call an $n$-chain of shape $D$ is a saturated chain of $n$ elements starting at 
$\nearrow\searrow$ and arriving at $D$.
Let $(\mathcal{D},\leq)$ denote this poset, the goal of this section is to prove that it is lattice.

Let $D$ be a Dyck path. Its {\it Dyck word} is the infinite binary word $(d_i)_{i\geq 1}$ obtained
from $D$ by replacing $\nearrow$ with $1$, $\searrow$ with $0$, and appending an infinite sequence of $0$'s.
Note that the rank of a Dyck path is the number of 1's in the corresponding Dyck word.
Though seemingly useless, appending the final sequence of $0$'s slightly simplifies some statements such as
the next proposition.

The height in a Dyck path $D$ after $i$ steps is the number of $1$'s minus the number of $0$'s among $d_1,\dots,d_i$.
We have $d_i\geq 0$ for $i\leq 2n$ where $n$ is the rank of the path and $d_i<0$ for $i>2n$.
We also define the height similarly for a binary word which is not {\it a priori} a Dyck word.

\begin{prop}
 Let $D,E\in\mathcal{D}$ and $(d_i)_{i\geq1}, (e_i)_{i\geq1}$ their Dyck words.
 We have $D\leq E$ if and only if $d_i\leq e_i $ for all $i$.
\end{prop}

\begin{proof}
It is clear that if $D_1$, $D_2$ are Dyck words, $D_1 \sqsubset D_2$ means 
that $D_2$ can be obtained from $D_1$ by changing a $0$ into a $1$.
So $D\sqsubset E$ implies $d_i\leq e_i $ for all $i$.
We deduce that $D\leq E$ also implies this condition.

It remains to show the other implication. So we suppose now $d_i\leq e_i $ for all $i$,
and we will show $D\leq E$ by induction on the difference of ranks.

The initial case is clear: if $D$ and $E$ have the same rank, we get $d_i=e_i$ for all $i$
so $D=E$ and in particular $D\leq E$.
Otherwise, let $M$ be the maximal integer such that $d_M<e_M$. In particular, $d_M=0$ and $e_M=1$.
We define $E'=(e'_i)$ by changing this $e_M=1$ into a $0$. 
If $E'$ is indeed a Dyck word, we have $D\leq E'$ by using the induction hypothesis.
Also $E'\sqsubset E$ by construction, so $E'\leq E$, and it follows $D\leq E$. 

So it remains to show that $E'$ is a Dyck word. If the difference of ranks is $1$, we have $E'=D$ 
so it is a Dyck word. So we can assume that the difference of ranks is at least 2. 
From $d_i \leq e_i$, we get that $h_i(E) - h_i(D)$ is increasing with $i$.
We have that $h_M(E) - h_M(D)$ is twice the difference of rank of $E$ and $D$, i.e., at least $4$.
So $h_M(E') - h_M(D)$ is at least $2$. So $h_i(E') = h_i(D) + 2 $ for $i\geq M$, so it is nonnegative as long
as $h_M(D) \geq -2 $. Now let $M'$ be the minimal integer with $h_{M'}(E')=0$. We easily obtain
$d_i=e_i=0$ for $i\geq M'$ and the fact that $E'$ is a Dyck word.
\end{proof}

\begin{prop}
$\mathcal{D}$ is a lattice. 
\end{prop}

\begin{proof}
Let $D$ and $E$ be two Dyck path. To show the existence of the join operation,
let us define a Dyck path $D \vee E$ by:
\[
  (D\vee E)_i = \max(d_i,e_i).
\]
Using the previous lemma,
it is clear that $D\leq D\vee E$, $E\leq D\vee E$, and $D\leq F$, $E\leq F$ implies
$D\vee E \leq F$. So we only have to show that $D\vee E$ is indeed a Dyck path with
this definition. Suppose for example that the length of $D$ is bigger than that of $E$.
We see from the definition that $D\vee E$ is obtained from $D$ by changing some steps
$\searrow$ into $\nearrow$, so it is clear that the new path cannot go below height $0$.

Now, we show the existence of the meet operation $D\wedge E$.
First, consider the sequence $c$ defined by
\[
  c_i = \min(d_i,e_i).
\]
It might not be a Dyck word, for example if $D=110010...$ and $E=101010...$ (where 
the dots represent the infinite sequence of $0$'s).
So, let us consider the smallest prefix $c'$ of $c$ containing more $0$'s than $1$'s.
We define $D\wedge E$ as the Dyck word $c'...$ (where the dots represent the infinite sequence of $0$'s).
It is clear that $D\wedge E \leq D$, $D\wedge E \leq E$. Suppose that $F\leq D$ and $F\leq E$ for some
Dyck path $F=(f_i)$. We have $f_i\leq c_i$, so the path $F$ goes below height $0$ after $|c'|$ steps. 
We deduce that $f_i=0$ when $i$ is greater than $|c'|$. So $F\leq D\wedge E$.
\end{proof} 

\section{Permutations and their profiles}

\label{sec:5}

As mentioned above, the number of $n$-chains of Dyck shapes is $n!$, and the goal of this section 
is to give a bijection with permutations. It is related with the notion of the {\it profile}
of a permutation as defined by Françon and Viennot \cite{francon}, i.e., a Dyck path encoding the location of 
peaks, valleys, double ascents and double descents, as defined below
(the term {\it shape} was maybe more close to the french ``forme'', but it was already in use in this paper). 
This notion of profile gives another interpretation of the relation $\sqsubset$ (addition of a ribbon) in Dyck paths,
giving a natural bijection between chains of Dyck shapes and permutations.

\begin{defi}
 Let $\sigma\in\mathfrak{S}_n$. We will take the convention that $\sigma_0=0$ and $\sigma_{n+1}=n+1$.
 The integer $\sigma_i\in\{1,\dots,n\}$ is called:
\begin{itemize}
 \item a {\it peak} if $\sigma_{i-1} < \sigma_i > \sigma_{i+1} $, 
 \item a {\it valley} if $\sigma_{i-1} > \sigma_i < \sigma_{i+1} $, 
 \item a {\it double ascent} if $\sigma_{i-1} < \sigma_i < \sigma_{i+1} $, 
 \item a {\it double descent} if $\sigma_{i-1} > \sigma_i > \sigma_{i+1} $.
\end{itemize}
\end{defi}

For example, if $\sigma=42135$, $1$ is a valley, $2$ is a double descent, $3$ and $5$ are double ascents, $4$ is a peak.
Note that in this definition, the value $\sigma_i$, not the index $i$, is a peak if $\sigma_{i-1} < \sigma_i > \sigma_{i+1} $
(and similarly for the other statistics). This is not coherent with the usual definition of a descent and ascent
(an index $i$ such that $\sigma_i > \sigma_{i+1}$ or $\sigma_i < \sigma_{i+1}$), but it is nonetheless standard.

\begin{defi}
The {\it profile} of a permutation $\sigma\in\mathfrak{S}_n$ is a Dyck path $\Delta(\sigma)$ of length $2n$, 
defined as follows. We see it as a binary word $w=w_1\cdots w_{2n}$ over the alphabet $\{\nearrow,\searrow\}$, such that
the factor $w_{2i-1}w_{2i}$ is :
\begin{itemize}
 \item $\nearrow\nearrow$ if $i$ is a valley of $\sigma$,
 \item $\searrow\searrow$ if $i$ is a peak of $\sigma$,
 \item $\nearrow\searrow$ if $i$ is a double ascent of $\sigma$,
 \item $\searrow\nearrow$ if $i$ is a double descent of $\sigma$.
\end{itemize}
\end{defi}

In the previous example $\sigma=42135$, we get 
$\Delta(\sigma) = \nearrow\nearrow \searrow\nearrow \nearrow\searrow \searrow\searrow \nearrow\searrow $
(which is as well the Dyck path represented in Figure~\ref{dyck}).

\begin{rema}
This notion of profile of a permutation and the associated Dyck path $\Delta(\sigma)$ has also been introduced
by Brändén and Leander \cite{brandenleander}, in slightly different terms.
\end{rema}

Let $\sigma\in\mathfrak{S}_n$ be a permutation, considered as a word $\sigma_1\cdots \sigma_n$. 
We define $\sigma'\in\mathfrak{S}_{n-1}$ as what we obtain after removing $n$.
For example, $(1457236)'=145236$.
More generally, by mimicking the notation for derivatives, we define $\sigma^{(0)}=\sigma$
and $\sigma^{(i+1)} = (\sigma^{(i)})'$.

The goal of this section is to prove:

\begin{theo} \label{bijchain}
The map 
\begin{equation}  \label{bijpermchain}
 \sigma \mapsto \big( \Delta(\sigma^{(n-1)}) ,\dots, \Delta(\sigma''),  \Delta(\sigma') , \Delta(\sigma) \big)
\end{equation}
defines a bijection between $\mathfrak{S}_n$ and $n$-chains of Dyck shapes.
\end{theo}

The inverse bijection will be described explicitly below. 
Anticipating a little, let us announce that the permutations
\[
 \tau = 513462 \quad \text{ and } \quad \upsilon = 5471326
\]
correspond to the chains of Dyck shapes in Figure~\ref{chain}.
We can check for example that
\begin{align*}
\Delta( \tau^{(0)} ) &= \Delta( 513462 ) = \nearrow\nearrow \nearrow\nearrow \nearrow\searrow \nearrow\searrow \searrow\searrow \searrow\searrow, \\
\Delta( \tau^{(1)} ) &= \Delta( 51342 ) = \nearrow\nearrow \nearrow\nearrow \nearrow\searrow \searrow\searrow \searrow\searrow, \\
\Delta( \tau^{(2)} ) &= \Delta( 1342 ) = \nearrow\searrow \nearrow\nearrow \nearrow\searrow \searrow\searrow, \\
\Delta( \tau^{(3)} ) &= \Delta( 132 ) =  \nearrow\searrow \nearrow\nearrow \searrow\searrow, \\
\Delta( \tau^{(4)} ) &= \Delta( 12 ) =   \nearrow\searrow \nearrow\searrow, \\
\Delta( \tau^{(5)} ) &= \Delta( 1 ) =    \nearrow\searrow, 
\end{align*}
which corresponds to the first chain of Dyck paths in Figure~\ref{chain}.

The proof of Theorem~\ref{bijchain} and the definition of the inverse bijection will follow from the lemma below.

\begin{lemm}
 Let $\sigma\in\mathfrak{S}_n$, then we have $\Delta(\sigma') \sqsubset \Delta(\sigma)$.
If $\Delta(\sigma')$ is a prefix of $\Delta(\sigma)$, then $\sigma$ is obtained from $\sigma'$ by inserting $n$ in 
the last position. Otherwise, let $k$ be the index of the first step where the paths $\Delta(\sigma')$ and $\Delta(\sigma)$ 
differ, then:
\begin{itemize}
 \item if $k$ is odd, $\sigma$ is obtained from $\sigma'$ by inserting $n$ just after $\frac{k+1}2$,
 \item if $k$ is even, $\sigma$ is obtained from $\sigma'$ by inserting $n$ just before $\frac k2$.
\end{itemize}
\end{lemm}

\begin{proof}
Let us compare $\Delta(\sigma')$ and $\Delta(\sigma)$. If $n$ is at the end of $\sigma$, then $\Delta(\sigma)$
is obtained from $\Delta(\sigma')$ by adding $\nearrow\searrow$ at the end, corresponding to the double ascent $n$.
So $\Delta(\sigma') \sqsubset \Delta(\sigma)$ (this is the addition of a ribbon with one cell).

Otherwise, $n$ is between two other integers, i.e., there is in $\sigma$ a factor $i,n,j$ (with the convention
that $i=0$ if $n$ is at the beginning of $\sigma$). 
We suppose first that $i<j$. It means that either 
\begin{itemize}
 \item $j$ is a double ascent in $\sigma'$ and a valley in $\sigma$, or
 \item $j$ is a peak in $\sigma'$ and a double descent in $\sigma$.
\end{itemize}
On the Dyck paths, it means that the $j$th pair of steps is either $\nearrow\searrow$ in $\sigma'$ and 
$\nearrow\nearrow$ in $\sigma$, or $\searrow\searrow$ in $\sigma'$ and $\searrow\nearrow$ in $\sigma$,
i.e., in each case the $(2j)$th step is $\searrow$ in $\sigma'$ and $\nearrow$ in $\sigma$.
It is easily checked that the other steps do not change. So $\Delta(\sigma') \sqsubset \Delta(\sigma)$.
Suppose then that $i>j$. It means that either 
\begin{itemize}
 \item $i$ is a double descent in $\sigma'$ and a valley in $\sigma$, or
 \item $i$ is a peak in $\sigma'$ and a double ascent in $\sigma$.
\end{itemize}
On the Dyck paths, it means that the $i$th pair of steps is either $\searrow\nearrow$ in $\sigma'$ and 
$\nearrow\nearrow$ in $\sigma$, or $\searrow\searrow$ in $\sigma'$ and $\nearrow\searrow$ in $\sigma$,
i.e., in each case the $(2i-1)$st step is $\searrow$ in $\sigma'$ and $\nearrow$ in $\sigma$.
It is easily checked that the other steps do not change. So $\Delta(\sigma') \sqsubset \Delta(\sigma)$.

Note that the cases $i<j$ and $i>j$ we considered can be distinguished via the parity of the 
changed step between $\Delta(\sigma')$ and $\Delta(\sigma)$. Thus we can finish the proof of
the proposition: in the first case, we can find $j$ from the index of the changed step and we know 
that $n$ is just before $j$, and in the second case, we can find $i$ from the index of the changed step
and we know that $n$ is just after $i$.
\end{proof}

The inverse bijection of \eqref{bijpermchain} can be described explicitly. Let $(D_1,\dots,D_n)$ be an $n$-chain of Dyck paths.
We build a permutation $\sigma\in\mathfrak{S}_n$ by inserting successively $1,2,\dots,n$.
Start from $1\in\mathfrak{S}_1$, corresponding to the 1-chain of Dyck paths $(\nearrow\searrow)$.
Suppose we have already inserted $1,\dots,j$, then we insert
$j+1$ as follows. Let $i\geq1$ be such that the leftmost cell of the ribbon $D_{j+1} / D_j $
is in the $i$th column. If $i=2j+1$, which means the ribbon $D_{j+1} / D_j $ contains a single square, 
then $j+1$ is inserted in last position. Otherwise:
\begin{itemize}
 \item if $i$ is even and $i=2i'$, then $j+1$ is inserted to the left of $i'$,
 \item if $i$ is odd and $i=2i'-1$, then $j+1$ is inserted to the right of $i'$.
\end{itemize}

Let us give the details for the first example of Figure~\ref{chain}.
In this case, we have 6 ribbons, the leftmost cells of the successive ribbons are in columns 1,3,4,5,2,7.
So, starting from the permutation $1$:
\begin{itemize}
 \item first $i=3$, so insert 2 in the last position and get $12$,
 \item then $i=4$, so insert 3 to the left of $2$ and get $132$,
 \item then $i=5$, so insert 4 to the right of $3$ and get $1342$,
 \item then $i=2$, so insert 5 to the left of $1$ and get $51342$,
 \item then $i=7$, so insert 6 to the right of $4$ and get $513462$,
\end{itemize}
Thus we get $\tau=513462$.

\section{Laguerre histories and Dyck tableaux}

\label{sec:6}

One result of Françon and Viennot \cite{francon} is that the number of permutations with a 
given profile $D$ is a product of integers easily described in terms of the Dyck path $D$.
This is done via a bijection between permutations and {\it Laguerre histories} (see \cite{viennot1}), 
the name being justified by the link with Laguerre orthogonal polynomials.
To relate this with Theorem~\ref{bijchain} in the previous section, 
we give a bijection between chains of Dyck shapes and these objects.
Note that Laguerre histories were previously related to the PASEP partition function and permutation 
tableaux in \cite{josuat}.

\begin{defi}
 A column of a Dyck path is {\it below a step} $\nearrow$ if
 its topmost cell is not contiguous to another cell to its North-West. 
 Alternatively, the $k$th column is {\it below a step} $\nearrow$ if the $k$th step of the path is $\nearrow$.
 Let $D$ be a Dyck path of length $2n$, then a 
 {\it Laguerre history} of shape $D$ is a filling of $D$ with $n$ dots such that there is one dot in each 
 column below a step $\nearrow$ (and no dot in the other columns).
\end{defi}

See Figure~\ref{hist} for some examples of Laguerre histories.

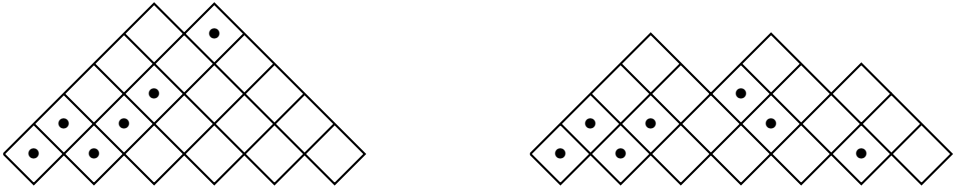
\begin{figure}[h!tp]   \centering  
 \begin{tikzpicture}[scale=0.4,thick]
   \draw (0,0) -- (5,5) -- (6,4) -- (7,5) -- (12,0) -- (11,-1) -- (6,4) -- (1,-1) -- (0,0);
   \draw (1,1) -- (3,-1) -- (8,4);
   \draw (2,2) -- (5,-1) -- (9,3);
   \draw (3,3) -- (7,-1) -- (10,2);
   \draw (4,4) -- (9,-1) -- (11,1);
   \node at (1,0) {$\bullet$};
   \node at (2,1) {$\bullet$};
   \node at (3,0) {$\bullet$};
   \node at (4,1) {$\bullet$};
   \node at (5,2) {$\bullet$};
   \node at (7,4) {$\bullet$};   
 \end{tikzpicture}
 \hspace{2cm}
 \begin{tikzpicture}[scale=0.4,thick]
   \draw (0,0) -- (4,4) -- (9,-1) -- (12,2) -- (14,0) -- (13,-1) -- (8,4) -- (3,-1) -- (1,1);
   \draw (0,0) -- (1,-1) -- (5,3);
   \draw (2,2) -- (5,-1) -- (9,3);
   \draw (3,3) -- (7,-1) -- (10,2);
   \draw (7,3) -- (11,-1) -- (13,1);
   \draw (10,2) -- (11,3) -- (12,2);
   \node at (1,0) {$\bullet$};
   \node at (2,1) {$\bullet$};
   \node at (3,0) {$\bullet$};
   \node at (4,1) {$\bullet$};
   \node at (7,2) {$\bullet$};
   \node at (8,1) {$\bullet$};
   \node at (11,0) {$\bullet$};   
  \end{tikzpicture} 
 \caption{Laguerre histories. \label{hist}}
\end{figure}

If a step $\nearrow$ in a Dyck path is from height $h$ to $h+1$, the column below it contains $\lfloor \frac h2 \rfloor+1$
cells (here the {\it height} of the step is seen as the $y$-coordinate in the path in $\mathbb{N}^2$ with two kinds of 
steps, not in the skew shape). This means that a Laguerre history is a Dyck path together with a choice among 
$\lfloor \frac h2 \rfloor+1$ possibilities for each step $\nearrow$ from height $h$ to $h+1$, i.e., we recover the notion 
of {\it subdivided Laguerre history} (see Viennot \cite[Chapter~2]{viennot1} and de Médicis and Viennot \cite[Section~4]{medicis}).
In particular, it follows from \cite[Section~4]{medicis} that there are $n!$ Laguerre histories of length $2n$.

\begin{prop}
  Let $D_1 \sqsubset \dots \sqsubset D_n$ be a chain of Dyck paths of shape $D=D_n$. We define a filling of $D$ by putting a dot 
  in the left extremity of each ribbon $D_{i+1}/D_i$. Then what we obtain is a Laguerre history, and this defines a bijection 
  between $n$-chains of Dyck shapes and Laguerre histories of the same shape $D$.
\end{prop}

\begin{proof}
We first prove that what we obtain is indeed a Laguerre history. Let $2\leq i \leq n$.
The Dyck path $D_i$ is obtained from $D_{i-1} \searrow \searrow$ by changing some step $\searrow$ (say, this is
the $k$th step) into a $\nearrow$. By construction, it means that the $i$th ribbon has its left extremity 
in the $k$th column of the Dyck path. We can deduce that there is exactly one dot in each column below a step
$\nearrow$ in $D$ at the end of the process.

Then, we define the inverse bijection. Let $H$ be a Laguerre history of shape $D$. To find what are the ribbons 
of the corresponding chain of Dyck paths, we proceed as follows. Start from the rightmost cell of the Dyck path, 
and follow the upper border of cells from right to left until arriving to a cell containing a dot; the path we just 
followed is the last added ribbon $D$, remove $D$ and do the same thing to find the other ribbons. 
It is easy to show that the two bijections are inverses of each other.
\end{proof}

For example, the bijection described in the previous proposition sends the chains in Figure~\ref{chain} to 
the Laguerre histories in Figure~\ref{hist}.

\begin{defi}[Aval et al. \cite{aval1}]
 Let $D$ be a Dyck shape of length $2n$. A {\it Dyck tableau of shape} $D$ is a filling of $D$ with $n$ dots
 such that there is one dot in each odd column.
\end{defi}

See Figure~\ref{tabl} for some examples of Dyck tableaux.

\begin{figure}[h!tp]   \centering
 \begin{tikzpicture}[scale=0.4,thick]
   \draw (0,0) -- (5,5) -- (6,4) -- (7,5) -- (12,0) -- (11,-1) -- (6,4) -- (1,-1) -- (0,0);
   \draw (1,1) -- (3,-1) -- (8,4);
   \draw (2,2) -- (5,-1) -- (9,3);
   \draw (3,3) -- (7,-1) -- (10,2);
   \draw (4,4) -- (9,-1) -- (11,1);
   \node at (1,0) {$\bullet$};
   \node at (3,0) {$\bullet$};
   \node at (5,2) {$\bullet$};
   \node at (7,4) {$\bullet$};
   \node at (9,0) {$\bullet$};
   \node at (11,0) {$\bullet$};   
 \end{tikzpicture}
 \hspace{2cm}
 \begin{tikzpicture}[scale=0.4,thick]
   \draw (0,0) -- (4,4) -- (9,-1) -- (12,2) -- (14,0) -- (13,-1) -- (8,4) -- (3,-1) -- (1,1);
   \draw (0,0) -- (1,-1) -- (5,3);
   \draw (2,2) -- (5,-1) -- (9,3);
   \draw (3,3) -- (7,-1) -- (10,2);
   \draw (7,3) -- (11,-1) -- (13,1);
   \node at (1,0) {$\bullet$};
   \node at (3,0) {$\bullet$};
   \node at (5,0) {$\bullet$};
   \node at (7,2) {$\bullet$};
   \node at (9,0) {$\bullet$};
   \node at (11,0) {$\bullet$};
   \node at (13,0) {$\bullet$};   
  \end{tikzpicture} 
 \caption{Dyck tableaux. \label{tabl}}
\end{figure}
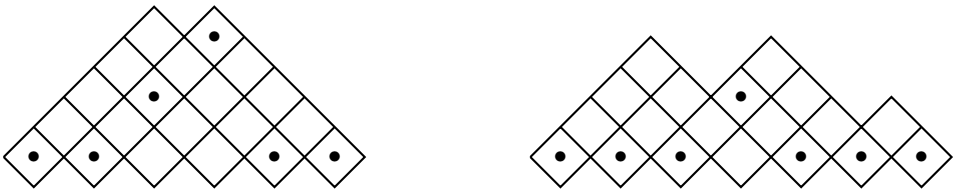

To see the equivalence between Dyck tableaux and our definition of Laguerre histories, we need the following lemma.

\begin{lemm}
 Let $D$ be a Dyck shape. Then there is a bijection between the columns of $D$ below a step $\nearrow$,
 and its odd columns. It is such that $c$ and $\kappa(c)$ contain the same number of cells.
\end{lemm}

\begin{proof}
The bijection $\kappa$ is constructed as follows. If a column $c$ is both odd and below a step $\nearrow$, then $\kappa(c)=c$.
This is always the case for the first column.
As for the other case, let $c$ be an even column below a step $\nearrow$.
It will be associated to $\kappa(c)$, an odd column of the same height as $c$ and which is not below a step $\nearrow$. 

We use the notion of {\it facing steps}. If we have a step $\nearrow$ in a Dyck path, from $(x,y)$ to $(x+1,y+1)$, its 
facing step is the step $\searrow$ from $(x',y+1)$ to $(x'+1,y)$ where $x'>x$ is minimal. Note that in this situation,
$x$ and $x'$ have different parity, which follows because a Dyck path only visits pairs $(i,j)$ with $i+j$ even.
Reciprocally, each step $\searrow$ is the facing step of a unique step $\nearrow$, than we can also call its facing step.

To define $\kappa(c)$, we follow the illustration in Figure~\ref{defkappa}.
The column $c$ is the below a step $\nearrow$, say $(x,y)$ to $(x+1,y+1)$. 
Note that $x+1$ is even, by the assumption that $c$ is an even column. 
And the number of cells of $c$ is $(y+1)/2$. In the example, $x=7$, $y=5$.
The step $\nearrow$ from $(x,y)$ to $(x+1,y+1)$ has a facing step $\searrow$ from $(x',y+1)$ to $(x'+1,y)$.
Then $\kappa(c)$ is defined as the $(x'+1)$st column, i.e., that below the end point of this facing step $\searrow$.

Since the notion of facing step gives a pairing between steps $\nearrow$ and $\searrow$, it is easily seen that
we can recover $c$ from $\kappa(c)$. By construction, $\kappa(c)$ is not below a step $\nearrow$, and $x'+1$, $y$
are odd. It follows that $\kappa(c)$ has the same number of cells as $c$. See Figure~\ref{defkappa}.
\end{proof}

\begin{figure}   \centering
   \begin{tikzpicture}[scale=0.4]
    \draw[dotted,thick] (-6,0) -- (18,0);
    \draw[fill=lightgray] (1,1) -- (2,2) -- (3,1) -- (2,0) -- (1,1);
    \draw[fill=lightgray] (1,3) -- (2,4) -- (3,3) -- (2,2) -- (1,3);
    \draw[fill=lightgray] (1,5) -- (2,6) -- (3,5) -- (2,4) -- (1,5);
    \draw  (7,1) -- (8,2) -- (9,1) -- (8,0) -- (7,1);
    \draw  (7,3) -- (8,4) -- (9,3) -- (8,2) -- (7,3);
    \draw  (7,5) -- (8,6) -- (9,5) -- (8,4) -- (7,5);
    \draw[fill=lightgray] (8,0) -- (9,1) -- (10,0) -- (9,-1) -- (8,0);
    \draw[fill=lightgray] (8,2) -- (9,3) -- (10,2) -- (9,1) -- (8,2);    
    \draw[fill=lightgray] (8,4) -- (9,5) -- (10,4) -- (9,3) -- (8,4);
    \draw[line width=0.6mm,->]  (1,5) -- (2,6);
    \draw[line width=0.6mm,->]  (8,6) -- (9,5);
    \draw (2,6) -- (3,7) -- (4,8) -- (5,7) -- (6,8) -- (7,7) -- (8,6);
    \draw (1,5) -- (0,4) -- (-1,5) -- (-2,4) -- (-3,3) -- (-4,2) -- (-5,1) -- (-6,0);
    \draw (10,4) -- (11,3) -- (12,2) -- (13,1) -- (14,2) -- (15,3) -- (16,2) -- (17,1) -- (18,0);
    \node at (2,-1) {$c$};
    \node at (9,-2) {$\kappa(c)$};    
   \end{tikzpicture}
   \caption{Definition of the bijection $\kappa$. \label{defkappa}}
\end{figure}
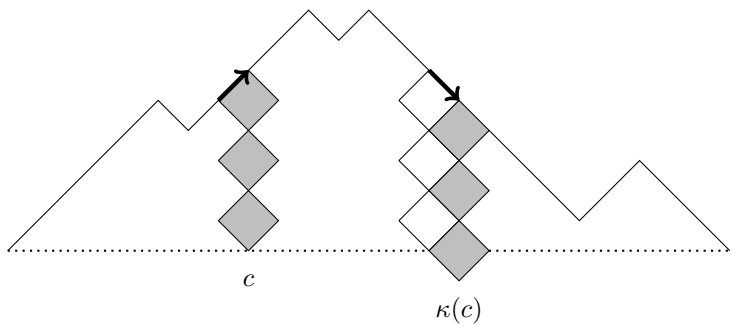


The previous lemma can be used to give a bijection between Laguerre histories of shape $D$ and Dyck tableaux of
the same shape. Let $H$ be a Laguerre history of shape $D$, we can define a Dyck tableau $T$ of shape $D$ 
by the following condition :
\begin{itemize}
 \item if a column $c$ in $H$ contains a dot in its $i$th cell (from bottom to top, for example), then 
       the column $\kappa(c)$ in $T$ also contains a dot in its $i$th cell.
\end{itemize}
This is clearly bijective.

For example, from the Laguerre histories in Figure~\ref{hist}, we obtain the Dyck tableaux in Figure~\ref{tabl}.
More precisely, the correspondence between columns can be illustrated by numbering dots, giving:
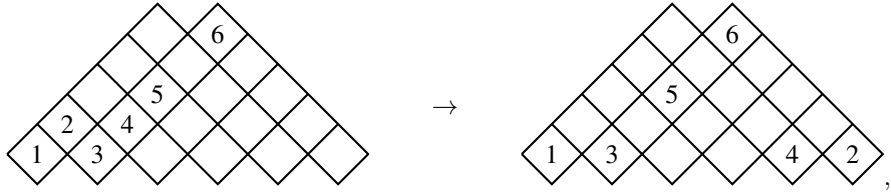
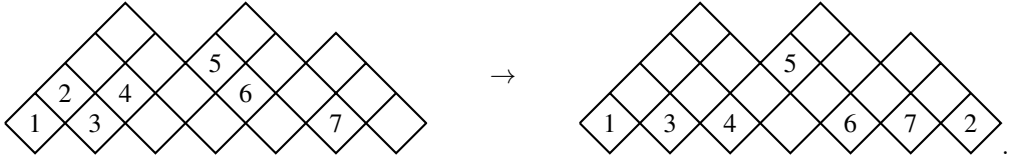
\[
   \begin{tikzpicture}[scale=0.4,thick]
   \draw (0,0) -- (5,5) -- (6,4) -- (7,5) -- (12,0) -- (11,-1) -- (6,4) -- (1,-1) -- (0,0);
   \draw (1,1) -- (3,-1) -- (8,4);
   \draw (2,2) -- (5,-1) -- (9,3);
   \draw (3,3) -- (7,-1) -- (10,2);
   \draw (4,4) -- (9,-1) -- (11,1);
   \node at (1,0) {1};
   \node at (2,1) {2};
   \node at (3,0) {3};
   \node at (4,1) {4};
   \node at (5,2) {5};
   \node at (7,4) {6};   
 \end{tikzpicture}
 \begin{tikzpicture}
     \draw[color=white] (0,0) rectangle (2,3);
     \node at (1,1) {$\to$};
 \end{tikzpicture}
 \begin{tikzpicture}[scale=0.4,thick]
   \draw (0,0) -- (5,5) -- (6,4) -- (7,5) -- (12,0) -- (11,-1) -- (6,4) -- (1,-1) -- (0,0);
   \draw (1,1) -- (3,-1) -- (8,4);
   \draw (2,2) -- (5,-1) -- (9,3);
   \draw (3,3) -- (7,-1) -- (10,2);
   \draw (4,4) -- (9,-1) -- (11,1);
   \node at (1,0) {1};
   \node at (3,0) {3};
   \node at (5,2) {5};
   \node at (7,4) {6};
   \node at (9,0) {4};
   \node at (11,0){2};   
 \end{tikzpicture},
\]
and
\[
 \begin{tikzpicture}[scale=0.4,thick]
   \draw (0,0) -- (4,4) -- (9,-1) -- (12,2) -- (14,0) -- (13,-1) -- (8,4) -- (3,-1) -- (1,1);
   \draw (0,0) -- (1,-1) -- (5,3);
   \draw (2,2) -- (5,-1) -- (9,3);
   \draw (3,3) -- (7,-1) -- (10,2);
   \draw (7,3) -- (11,-1) -- (13,1);
   \draw (10,2) -- (11,3) -- (12,2);   
   \node at (1,0) {1};
   \node at (2,1) {2};
   \node at (3,0) {3};
   \node at (4,1) {4};
   \node at (7,2) {5};
   \node at (8,1) {6};
   \node at (11,0){7};   
  \end{tikzpicture} 
 \begin{tikzpicture}
     \draw[color=white] (0,0) rectangle (2,3);
     \node at (1,1) {$\to$};
 \end{tikzpicture}
 \begin{tikzpicture}[scale=0.4,thick]
   \draw (0,0) -- (4,4) -- (9,-1) -- (12,2) -- (14,0) -- (13,-1) -- (8,4) -- (3,-1) -- (1,1);
   \draw (0,0) -- (1,-1) -- (5,3);
   \draw (2,2) -- (5,-1) -- (9,3);
   \draw (3,3) -- (7,-1) -- (10,2);
   \draw (7,3) -- (11,-1) -- (13,1);
   \draw (10,2) -- (11,3) -- (12,2);
   \node at (1,0) {1};
   \node at (3,0) {3};
   \node at (5,0) {4};
   \node at (7,2) {5};
   \node at (9,0) {6};
   \node at (11,0) {7};
   \node at (13,0) {2};   
  \end{tikzpicture}.
\]

\section{A generalization}

\label{sec:7}

In this section, we consider a generalization of stammering tableaux where we do not require 
$\lambda^{(0)}=\lambda^{(3n)}=\varnothing$ anymore. Let $T^{(n)}_{\mu,\nu}$ be the number of 
such tableaux $\lambda^{(0)},\dots,\lambda^{(3n)}$ with $\lambda^{(0)}=\mu$ and  $\lambda^{(3n)}=\nu$.
We don't know how to compute this number in general, but we have a simple
answer when either $\mu$ or $\nu$ is empty.

Also we need to consider a generalization of rook placements in $2\delta_n$, so in this section we do 
not require that there is exactly one dot per row.

\begin{lemm}
The number of fillings of $2\delta_n$ with $k$ dots, such that there is at most one dot per row and at most one per column, is  
\[
  \frac{(n+1)!}{(n-k+1)!}  \binom{n}{k}.
\]
\end{lemm}

\begin{proof}
 These numbers satisfy the recursion
\[
  a_{n,k} = a_{n-1,k} + (2n-k+1) a_{n-1,k-1},
\]
which can be checked either on the definition with rook placements or on the given formula.
The initial conditions are $a_{0,0}=1$, $a_{n,k}=0$ if $n<0$ or $k<0$ or $k>n$.
\end{proof}

\begin{prop} Let $\lambda\in\mathcal{Y}$ with $|\lambda|=k$, and let $f_\lambda$ be the number of standard tableaux of shape $\lambda$. Then we have
\[
  T^{(n)}_{\varnothing,\lambda} = (n+1)! \binom{n}{k} f_{\lambda}
\]
and 
\[
  T^{(n)}_{\lambda,\varnothing} = \frac{(n+1)!}{(k+1)!} \binom{n}{k} f_{\lambda}.
\]
\end{prop}

\begin{proof}
The idea is to use growth diagrams to build a bijection between these generalized stammering tableaux and 
rook placements in the double staircase together with some additional data. Note that we will here use the
{\it reverse local rules} of growth diagrams (\cite[Section~2.1.8]{langer}). Using the notation of Figure~\ref{growth}, 
it means that if we know $\mu$, $\nu$, $\rho$, we can find $\lambda$ and whether the cell contains a dot or not.

Let us begin with the first equality. 
Let $\lambda^{(0)},\dots,\lambda^{(3n)}$ with $\lambda^{(0)}=\varnothing$ and  $\lambda^{(3n)}=\lambda$.
We can put this sequence of partitions along the North-East border of $2\delta_n$ from the top left corner to the 
bottom right corner, and use the reverse local rules of growth diagrams to complete the picture, i.e.,
fill some cells with a dot and assign a Young diagram to each vertex. We obtain
a rook placement $R$ in $2\delta_n$ with $n$ dots, and the bottom border contains a sequence $\rho^{(0)},\dots,\rho^{(2n)}$
such that $\rho^{(i-1)}=\rho^{(i)}$ if the $i$th column of $R$ contains a dot, 
$\rho^{(i-1)}=\rho^{(i)}$ or $\rho^{(i-1)} \lessdot \rho^{(i)}$ otherwise.
This $\rho^{(0)},\dots,\rho^{(2n)}$ can be encoded by a standard tableau of shape $\lambda$ (describing
the sequence of shapes without repetition) and the choice of $k$ columns among the $n$ columns without 
a dot (describing what are the indices $i$ such that $\rho^{(i-1)} \lessdot \rho^{(i)}$).
So there are $\binom nk f_\lambda$ such sequences for a given rook placement.
Since there are $(n+1)!$ rook placements in $2\delta_n$, we obtain $T^{(n)}_{\varnothing,\lambda} = (n+1)! \binom{n}{k} f_{\lambda}$.

The second result is similar but starting with the assumption $\lambda^{(0)}=\lambda$ and  $\lambda^{(3n)}=\varnothing$.
We put this sequence along the North-East border of $2\delta_n$, and the reverse local rule of growth diagrams are performed.
We obtain a partial rook placement $R$ in $2\delta_n$ with $n-k$ dots, and the left border of $2\delta_n$ is a sequence
$\rho^{(0)},\dots,\rho^{(n)}$ such that $\rho^{(i-1)}=\rho^{(i)}$ if the $i$th row of $R$ contains a dot, and
$\rho^{(i-1)} \lessdot \rho^{(i)}$ otherwise. Using the previous lemma, we obtain 
$T^{(n)}_{\lambda,\varnothing} = \frac{(n+1)!}{(k+1)!} \binom{n}{k} f_{\lambda}.$
\end{proof}

To illustrate the first identity, if we start from \Yboxdim{8pt}
\[
  \big( \varnothing , \;
  \yng(1) \:, \;
  \yng(1,1) \:; \;
  \yng(1) \:, \;
  \yng(2) \:, \;
  \yng(2) \:; \;
  \yng(1) \:, \;
  \yng(1,1) \:, \;
  \yng(1,2) \:; \;
  \yng(2) \:, \;
  \yng(1,2) \:, \;
  \yng(2,2) \:; \;
  \yng(2) \:, \;
  \yng(1,2) \:, \;
  \yng(2,2) \:; \;
  \yng(1,2) \big),
\]
the growth diagram is in Figure~\ref{reversediagram}.
Here $n=4$ and $k=3$. We get a rook placement in $2\delta_4$, and the bottom line 
\[
   \varnothing \:, \;
   \varnothing \:, \;
   \varnothing \:, \;
   \varnothing \:, \;
   \varnothing \:, \;
   \yng(1) \:, \;
   \yng(1) \:, \;
   \yng(1,1) \:, \;
   \yng(1,2).
\]
This sequence of Young diagrams can be encoded in the standard tableau \Yboxdim{11pt}
\[ 
  \text{\tiny \young(2,13)} \:,
\]
together with the indices $i=5,7,9$ such that the $i$th column in the growth diagram 
has a change in its bottom line.

\Yboxdim{6pt}

\begin{figure}[h!tp]   \centering
  \begin{tikzpicture}[scale=0.64,thick]
    \draw[dotted,gray] (0,0) grid (2,4);
    \draw[dotted,gray] (2,0) grid (4,3);
    \draw[dotted,gray] (4,0) grid (6,2);
    \draw[dotted,gray] (6,0) grid (8,1);
    \node at (0.5,3.5) {$\bullet$};
    \node at (2.5,2.5) {$\bullet$};
    \node at (1.5,1.5) {$\bullet$};
    \node at (5.5,0.5) {$\bullet$};
    \node at (0,4) {$\varnothing$};
    \node at (1,4) {\yng(1)};
    \node at (2,4) {\yng(1,1)};
    \node at (2,3) {\yng(1)};
    \node at (3,3) {\yng(2)};
    \node at (4,3) {\yng(2)};
    \node at (4,2) {\yng(1)};
    \node at (5,2) {\yng(1,1)};
    \node at (6,2) {\yng(1,2)};
    \node at (6,1) {\yng(2)};
    \node at (7,1) {\yng(1,2)};
    \node at (8,1) {\yng(2,2)};
    \node at (8,0) {\yng(1,2)};
    \node at (1,3) {$\varnothing$};
    \node at (0,3) {$\varnothing$};
    \node at (3,2) {\yng(1)};
    \node at (2,2) {\yng(1)};
    \node at (1,2) {$\varnothing$};
    \node at (0,2) {$\varnothing$};
    \node at (5,1) {\yng(1)};
    \node at (4,1) {$\varnothing$};
    \node at (3,1) {$\varnothing$};
    \node at (2,1) {$\varnothing$};
    \node at (1,1) {$\varnothing$};
    \node at (0,1) {$\varnothing$};
    \node at (7,0) {\yng(1,1)};
    \node at (6,0) {\yng(1)};
    \node at (5,0) {\yng(1)};
    \node at (4,0) {$\varnothing$};
    \node at (3,0) {$\varnothing$};
    \node at (2,0) {$\varnothing$};
    \node at (1,0) {$\varnothing$};
    \node at (0,0) {$\varnothing$};
  \end{tikzpicture}
 \caption{A growth diagram illustrating reverse local rules. \label{reversediagram}}
\end{figure}

And if we start from \Yboxdim{8pt}
\[
  \yng(1,2) \:, \;
  \yng(1,2) \:, \;
  \yng(1,2) \:; \;
  \yng(2) \:, \;
  \yng(3) \:, \;
  \yng(3) \:; \;
  \yng(2) \:, \;
  \yng(2) \:, \;
  \yng(1,2) \:; \;
  \yng(1,1) \:, \;
  \yng(1,1) \:, \;
  \yng(1,1) \:; \;
  \yng(1) \:, \;
  \yng(1) \:, \;
  \yng(1) \:; \;
  \varnothing \:,
\]
the growth diagram is in Figure~\ref{reversediagram2}. Here $n=5$ and $k=3$. We get a rook placement in $2\delta_5$ with $2$ rooks, 
together with the sequence of Young diagrams in the rightmost line:
\[
  \varnothing \:, \; \varnothing \:, \; \yng(1) \:, \; \yng(2) \:, \; \yng(2) \:, \; \yng(1,2) \:,
\]
which can be encoded in the standard tableau \Yboxdim{11pt}
\[
  \text{\tiny \young(3,12)} \:.
\]
Note that we do not need to keep track of when repetitions occur in this sequence, because these
correspond to rows containing a dot in the growth diagram.

\begin{figure}[h!tp]  \centering \Yboxdim{6pt}  
  \begin{tikzpicture}[scale=0.64,thick]
    \draw[dotted,gray] (0,0) grid (2,5);
    \draw[dotted,gray] (2,0) grid (4,4);
    \draw[dotted,gray] (4,0) grid (6,3);
    \draw[dotted,gray] (6,0) grid (8,2);
    \draw[dotted,gray] (6,0) grid (10,1);
    \node at (2.5,3.5) {$\bullet$};
    \node at (5.5,0.5) {$\bullet$};
    \node at (0,5) {\yng(1,2)};
    \node at (1,5) {\yng(1,2)};
    \node at (2,5) {\yng(1,2)};
    \node at (2,4) {\yng(2)};
    \node at (3,4) {\yng(3)};
    \node at (4,4) {\yng(3)};
    \node at (4,3) {\yng(2)};
    \node at (5,3) {\yng(2)};
    \node at (6,3) {\yng(1,2)};
    \node at (6,2) {\yng(1,1)};
    \node at (7,2) {\yng(1,1)};
    \node at (8,2) {\yng(1,1)};
    \node at (8,1) {\yng(1)};
    \node at (9,1) {\yng(1)};
    \node at (10,1) {\yng(1)};
    \node at (10,0) {$\varnothing$};
    \node at (1,4) {\yng(2)};
    \node at (0,4) {\yng(2)};
    \node at (3,3) {\yng(2)};
    \node at (2,3) {\yng(2)};
    \node at (1,3) {\yng(2)};
    \node at (0,3) {\yng(2)};
    \node at (5,2) {\yng(1)};
    \node at (4,2) {\yng(1)};
    \node at (3,2) {\yng(1)};
    \node at (2,2) {\yng(1)};
    \node at (1,2) {\yng(1)};
    \node at (0,2) {\yng(1)};
    \node at (7,1) {\yng(1)};
    \node at (6,1) {\yng(1)};
    \node at (5,1) {$\varnothing$};
    \node at (4,1) {$\varnothing$};
    \node at (3,1) {$\varnothing$};
    \node at (2,1) {$\varnothing$};
    \node at (1,1) {$\varnothing$};
    \node at (0,1) {$\varnothing$};
    \node at (9,0) {$\varnothing$};
    \node at (8,0) {$\varnothing$};
    \node at (7,0) {$\varnothing$};
    \node at (6,0) {$\varnothing$};
    \node at (5,0) {$\varnothing$};
    \node at (4,0) {$\varnothing$};
    \node at (3,0) {$\varnothing$};
    \node at (2,0) {$\varnothing$};
    \node at (1,0) {$\varnothing$};
    \node at (0,0) {$\varnothing$};    
  \end{tikzpicture}
 \caption{A growth diagram illustrating reverse local rules. \label{reversediagram2}}
\end{figure}

\section{Open problems}

\label{sec:8}

The poset introduced in Section~\ref{sec:4} is reminicent of the notion of Fomin's 
dual graded graphs \cite{fomin2}, since there are $n!$ paths from the minimal element to some 
element of rank $n$. So the problem is to find some dual order relation $\preccurlyeq$
on the set of Dyck paths of all possible lengths, so that we have the commutation relation
$DU-UD=I$ (or some variation) for the up and down operators as in \eqref{duud} where $U$ is 
defined with the cover relation of $\leq$ and $D$ is defined with the cover relation of 
$\preccurlyeq$. See \cite{fomin2} for details. Note that $\preccurlyeq$ must be a tree
in this situation.  If there is such a construction, it might be related with a 
Hopf algebra of Dyck paths.  So, defining a relevant coproduct on the algebra
of Dyck path of Brändén and Leander \cite{brandenleander} is a related question.

It is well known that standard tableaux are related with representations of the symmetric 
groups. Oscillating tableaux are related with that of Brauer algebras \cite{barceloram}, 
and vacillating tableaux are related with that of partition algebras \cite{halversonram}.
It would be interesting to see if our stammering tableaux have any similar algebraic meaning.

On the purely combinatorial level, there are a lot of enumerative or bijective results
on Dyck tableaux, tree-like tableaux, permutation tableaux, and links with permutation statistics,
$q$-Eulerian numbers, etc. See for example \cite{aval1,aval2,josuat,viennot}. 
It would be interesting to see if the combinatorial objects presented in this paper are helpful 
on these topics. 


\setlength{\parindent}{0mm}


\begin{thebibliography}{999}

\bibitem{barceloram}
 \textsc{H. Barcelo and A. Ram}:
 Combinatorial representation theory. 
 New perspectives in algebraic combinatorics (Berkeley, CA, 1996–97), 23--90, 
 Math. Sci. Res. Inst. Publ. 38, Cambridge Univ. Press, Cambridge, 1999.

\bibitem{aval1}
 \textsc{J.-C. Aval, A. Boussicault and S. Dasse-Hartaut}:
 {\it Dyck tableaux.}
 Theoret. Comput. Sci. 502 (2013), 195--209.

\bibitem{aval2}
 \textsc{J.-C. Aval, A. Boussicault and P. Nadeau}:
 {\it Tree-like tableaux.}
 FPSAC'2011, Reykjavik, Island, DMTCS Proceedings (2011). 


\bibitem{brandenleander}
 \textsc{P. Brändén, M. Leander}:
 {\it Multivariate $P$-Eulerian polynomial.}
 Preprint. 
 
 \href{https://arxiv.org/abs/1604.04140}{arXiv:1604.04140}.
 
\bibitem{chen}
 \textsc{W.Y.C. Chen, E.Y.P. Deng, R.R.X. Du, R.P. Stanley and C.H. Yan}:
 {\it Crossings and nestings of matchings and partitions.}
 Trans. Amer. Math. Soc. 359 (2007), 1555--1575. 

\bibitem{corteel}
 \textsc{S. Corteel and L. K. Williams}:
 {\it Tableaux combinatorics for the asymmetric exclusion process.}
 Adv. in Appl. Math. 39(3) (2007), 293--310.

\bibitem{derrida}
 \textsc{B. Derrida, M. Evans, V. Hakim and V. Pasquier}: 
 {\it Exact solution of a 1D asymmetric exclusion model using a matrix formulation.}
 J. Phys. A: Math. Gen. 26 (1993), 1493–1517.

\bibitem{fomin}
 \textsc{S.V. Fomin}:
 {\it The generalized Robinson-Schensted-Knuth correspondence.}
 (Russian) Zap. Nauchn. Sem. Leningrad. Otdel. Mat. Inst. Steklov. (LOMI) 155 (1986), 
 translation in J. Soviet Math. 41 (1988), no. 2, 979--991.

\bibitem{fomin2}
 \textsc{S.V. Fomin}:
 {\it Duality of graded graphs.}
 J. Algebraic Combin. 3(4) (1994), 357--404. 

\bibitem{francon}
 \textsc{J. Françon and G. Viennot}: 
 {\it Permutations selon leurs pics, creux, doubles montées et double descentes, nombres d'Euler et nombres de Genocchi.}
 Discrete Math. 28(1) (1979), 21--35.

\bibitem{josuat}
 \textsc{M. Josuat-Vergès}:
 {\it Combinatorics of the three-parameter PASEP partition function.}
 El. J. of Combin. 18(1) (2011), Article P22. 

\bibitem{halversonram}
 \textsc{T. Halverson and A. Ram}:
 {\it Partition algebras.}
 European J. Combin. 26(6) (2005), 869--921.

\bibitem{krattenthaler}
 \textsc{C. Krattenthaler}: 
 {\it Growth diagrams, and increasing and decreasing chains in fillings of Ferrers shapes.}
 Adv. in Appl. Math. 37(3) (2006), 404--431.

\bibitem{langer} 
 \textsc{R. Langer}:
 {\it Cylindric plane partitions, lambda determinant, commutators in semicircular systems.}
 Phd thesis, Université Paris-Est Marne-la-Vallée, 2013.
 
 \url{https://pastel.archives-ouvertes.fr/tel-01327405}
 
\bibitem{medicis}
 \textsc{A. de Médicis and X.G. Viennot}: 
 {\it Moments de polynômes de $q$-Laguerre et la bijection de Foata-Zeilberger.}
 Adv. in Appl. Math. 15 (1994), 262--304.

\bibitem{roby}
 \textsc{T.W. Roby}:
 {\it Applications and extensions of Fomin's generalization of the Robinson-Schensted correspondence to differential posets.}
 Phd. Thesis, M.I.T., Cambridge, 1991.

\bibitem{stanley} 
 \textsc{R.P. Stanley}: 
 {\it Differential posets.}
 J. Am. Math. Soc. 1 (1988), 919--961.

\bibitem{stanley2}
 \textsc{R.P. Stanley}:
 Enumerative Combinatorics, Volume 2. Cambridge University Press, 1999.
 
 
\bibitem{varvak}
 \textsc{A. Varvak}:
 {\it Rook numbers and the normal ordering problem.}
 J. Combin. Theory Ser. A 112(2) (2005), 292--307.

\bibitem{viennot1}
 \textsc{G. Viennot}: 
 {\it Une théorie combinatoire des polynômes orthogonaux.}
 Notes de cours, UQÀM, Montréal, 1984.


 \url{http://www.xavierviennot.org/xavier/polynomes_orthogonaux.html}
 
\bibitem{viennot2}
 \textsc{G. Viennot}:
 {\it Une forme géométrique de la correspondance de Robinson-Schensted.}
 In « Combinatoire et représentation du groupe symétrique », 
 Lecture Notes in Math. 579, ed. D. Foata, Springer-Verlag, Berlin, 1978, pp. 29--58.

\bibitem{viennot}
 \textsc{X. Viennot}: 
 {\it Alternative tableaux, permutations and partially asymmetric exclusion process.}
 Talk at the Isaac Newton institute, April 2007.
 
 \url{https://www.newton.ac.uk/seminar/20080423140015001}
 


\end{thebibliography}
\end{document}